\documentclass[reqno,11pt]{amsart}

\usepackage{times,amsmath,cancel,stmaryrd,graphicx}
\usepackage{amsfonts,enumitem,amssymb,color,mathrsfs}
\usepackage[colorlinks=true]{hyperref}
\usepackage{cleveref,enumerate}
\usepackage{lmodern}
\usepackage[dvipsnames]{xcolor}
\usepackage{comment}
\usepackage{tikz}
\usetikzlibrary{arrows.meta,calc,decorations.markings}
\usepackage{cite}

\usepackage{float}

\newtheorem{thm}{Theorem}[section]
\newtheorem{lem}[thm]{Lemma}

\newtheorem{prop}[thm]{Proposition}

\numberwithin{equation}{section}

\newcommand{\aequation}{\renewcommand{\theequation}{\mbox{A.\arabic{equation}}}}

\newcommand{\nequation}{\setcounter{equation}{0}}

\newcommand{\R}{\mathbb{R}}

\newcommand{\E}{\mathbb{E}}

\newcommand{\rd}{\mathrm{d}}

\newcommand{\dhr}{\mathrel{\lhook\joinrel\relbar\kern-.8ex\joinrel\lhook\joinrel\rightarrow}}

\allowdisplaybreaks

\begin{document}

\title[Equilibria in a Biofilm Model]{Stability of Equilibria in a Biofilm Reactor Model with Wall Attachment and Thermodynamic Growth Inhibition}


%
\author{Katerina Nik}
\address{King Abdullah University of Science and Technology (KAUST)\\
CEMSE Division\\
Thuwal 23955-6900\\
Saudi Arabia}
\email{katerina.nik@kaust.edu.sa}

\author{Christoph Walker}
\address{Leibniz Universit\"at Hannover\\
Institut f\"ur Angewandte Mathematik\\
Welfengarten 1\\
30167 Hannover\\
Germany}
\email{walker@ifam.uni-hannover.de}
\date{\today}

\begin{abstract}
The dynamics of a mathematical model for a chemostat-type reactor is investigated. The model describes the temporal evolution of suspended and wall-attached bacterial populations, with the latter represented as a one-dimensional biofilm, subject to a non-reproducing growth-limiting substrate and a reaction product formed through bacterial utilization of the substrate. In particular, it is shown that, in the regime where the trivial (washout) equilibrium is unstable, there exists a unique nontrivial equilibrium that is locally asymptotically stable. Under slightly stronger assumptions, uniform persistence and global asymptotic stability of the nontrivial equilibrium are established.
\end{abstract}
%
\keywords{Biofilm, uniform persistence, nontrivial equilibrium, global stability}
\subjclass[2020]{37N25,92D25,34C11}
\maketitle
\section{Introduction}

Biofilms are structured microbial communities that develop when suspended microorganisms attach to surfaces or aggregate with other cells. In biotechnological reactors, microorganisms transform chemical compounds through biochemical processes, thereby enabling growth and substrate degradation. Depending on the reactor environment, microbial biomass can exist either as suspended planktonic cells in the liquid phase or as attached microbial communities within a biofilm. These two forms differ substantially in their local environmental conditions, transport limitations, and metabolic activity.
Suspended and attached microbial populations can be considered as interacting compartments coupled through attachment and detachment processes: planktonic cells may be transferred from the bulk liquid phase to established biofilms, while sessile cells can detach and re-enter the suspended phase. Such exchange mechanisms strongly influence population dynamics, community composition, and the long-term behavior of microbial systems. We refer to~\cite{Freter,Jones03,WannerGujer86,WannerEtAl06,KD10} for established biofilm modeling frameworks.\\

In this work, we investigate a mathematical model of a chemostat-type reactor containing suspended and wall-attached bacterial populations while accounting for thermodynamic growth inhibition. The model was originally introduced in \cite{GHE21} and extends the biofilm reactor model developed in \cite{MasicEberl12}, which, in turn, is based on the classical Freter model~\cite{Freter,Jones03,BallykJonesSmith,Stemmons} and includes additional dynamical features. It describes the temporal evolution of suspended and attached bacterial populations, with the latter represented as a one-dimensional bacterial biofilm. In addition, the model accounts for two chemical species: a non-reproducing, growth-limiting substrate (propionate) and a reaction product (acetate) formed through bacterial utilization of the substrate.\\

More precisely, in this work we focus on the system of boundary value problems
\begin{subequations}\label{C}
\begin{align}
\kappa_1\partial_z^2 c_1(t,z)&=r\big(c_1(t,z),c_2(t,z)\big)\,, \qquad\qquad\qquad\qquad 0<z<h(t)\,, \label{C1a}\\
\partial_z c_1(t,0)&=0\,, \qquad  c_1\big(t,h(t)\big)=S_1(t)\,,\label{C2a}\\
\kappa_2\partial_z^2 c_2(t,z)&=-r\big(c_1(t,z),c_2(t,z)\big)\,,\qquad\qquad\qquad\qquad 0<z<h(t)\,,  \label{C1b}\\
\partial_z c_2(t,0)&=0\,, \qquad  c_2\big(t,h(t)\big)=S_2(t)\,,\label{C2b}
\end{align}
for the unknowns $c_1=c_1(t,z)$ and $c_2=c_2(t,z)$, representing the propionate and acetate concentrations within the biofilm at location $z$, respectively, 
the total biofilm thickness $h=h(t)$,  as well as the propionate concentration 
$S_1=S_1(t)$ and the acetate concentration $S_2=S_2(t)$ in the aqueous phase. The latter three unknowns are coupled to  the suspended biomass $Q=Q(t)$ in the aqueous phase and governed by the system of ordinary differential equations
\begin{align}
h'(t)&=k_1\int_0^{h(t)}  r\big(c_1(t,z),c_2(t,z)\big) \rd z -bh(t)+\frac{\alpha}{\beta} Q(t)-d\big(h(t)\big)h(t)\,, \label{C3}\\
S_1'(t)&=D\big(S^*-S_1(t)\big)-r\big(S_1(t),S_2(t)\big)Q(t)-\beta\int_0^{h(t)} r\big(c_1(t,z),c_2(t,z)\big) \rd z \,, \label{C4a}\\
S_2'(t)&=-DS_2(t)+r\big(S_1(t),S_2(t)\big)Q(t)+\beta\int_0^{h(t)} r\big(c_1(t,z),c_2(t,z)\big) \rd z \,, \label{C4b}\\
Q'(t)&= \big(k_1r\big(S_1(t),S_2(t)\big)-k_2-\alpha\big)Q(t)+\beta d\big(h(t)\big)h(t)\,, \label{C5}
\end{align}
for $t>0$ subject to the nonnegative initial conditions
\begin{align}
h(0)&=h^0\,,\quad S_1(0)=S_1^0\,,\quad S_2(0)=S_2^0\,,\quad Q(0)=Q^0 \label{C7}\,.
\end{align}
\end{subequations}
The functions $r$ and $d$ represent the thermodynamic growth function of the bacteria and the detachment rate of bacteria from the biofilm into the aqueous phase, respectively. Moreover, $S^*>0$ denotes the constant influent propionate concentration, and $b, D, k_1, k_2, \kappa_1, \kappa_2, \alpha, \beta$ are positive, re-scaled constants (with $\kappa_2>\kappa_1$ being associated with acetate and propionate diffusion, respectively).  The situation is illustrated in Figure~\ref{fig:reactor}.
We refer to \cite{GHE21} (and the references therein) for a more detailed account of the modeling aspects and the precise meanings of the constants.\\

The global well-posedness of~\eqref{C} was established in~\cite{GHE21} for a particular choice of the thermodynamic growth function $r$ (see~\eqref{R:kinetics} below). In the same work, the local stability of the washout equilibrium $(h,S_1,S_2,Q)=(0,S^*,0,0)$ was analyzed, together with further qualitative properties of the long-term dynamics, supported by numerical simulations. Based on computational evidence, the existence of a unique stable persistence equilibrium was conjectured in \cite[Conjecture 1]{GHE21}, although a rigorous proof was not provided. It is worth pointing out that the complexity of the model makes analytical results concerning the persistence equilibrium challenging to obtain.

The main objective of the present work is to establish rigorously the existence of a (unique) persistence equilibrium and to analyze its stability in the regime where the washout equilibrium is unstable. To this end, we employ a shooting argument to prove existence and uniqueness of the nontrivial equilibrium. 
This part builds upon our previous works \cite{GW_QAM,NW26}, where related one-dimensional biofilm models were investigated.
 Furthermore, we  derive its local asymptotic stability by means of the principle of linearized stability. Under slightly stronger assumptions, we also establish its global stability relying on uniform persistence theory for semiflows \cite{HaleWaltman1989,SmithThieme2011} and
 using the geometric approach of  Li and Muldowney~\cite{LiMuldowney1996} in order to exclude non-constant periodic orbits  as well as homoclinic and heteroclinic orbits in the dynamics.\\

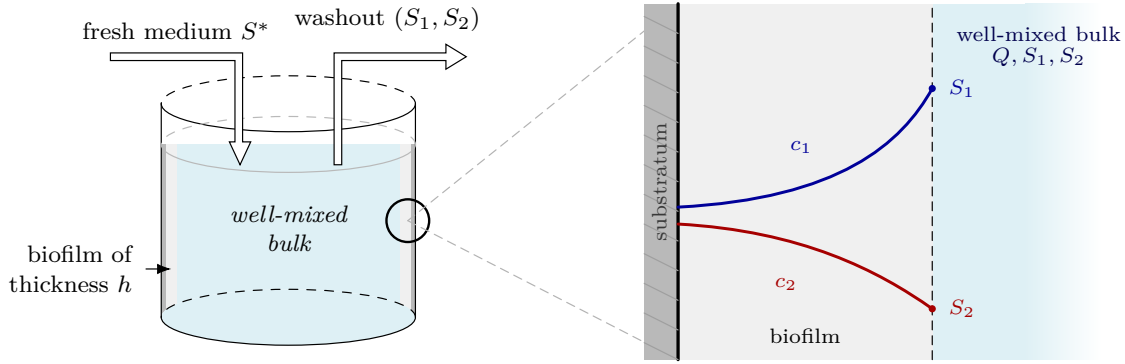
\begin{figure}[h]
\centering
\definecolor{bulkaqua}{RGB}{222,240,245}
\resizebox{\textwidth}{!}{%
\begin{tikzpicture}[line cap=round,line join=round,>=Latex,font=\small]
\begin{scope}[shift={(-6.2,0)},scale=0.6]
  \def\W{5.0}\def\H{4.2}\def\E{0.55}\def\level{0.8}\def\film{0.30}
  \def\wall{0.10}\def\bio{0.20}
  \pgfmathsetmacro{\yLvl}{\H-\level}
  \def\Cyl{(-\W/2,\H) arc[start angle=180,end angle=360,x radius=\W/2,y radius=\E] -- (\W/2,0) arc[start angle=0,end angle=180,x radius=\W/2,y radius=\E] -- cycle}
  \begin{scope}
    \clip \Cyl;
    \fill[bulkaqua] (-\W/2,-1) rectangle (\W/2,\yLvl);
    \fill[gray!55] (-\W/2,-1) rectangle (-\W/2+\wall,\yLvl);
    \fill[gray!12] (-\W/2+\wall,-1) rectangle (-\W/2+\wall+\bio,\yLvl);
    \fill[gray!55] (\W/2-\wall,-1) rectangle (\W/2,\yLvl);
    \fill[gray!12] (\W/2-\wall-\bio,-1) rectangle (\W/2-\wall,\yLvl);
  \end{scope}
  \fill[bulkaqua] (0,0) ellipse[x radius=\W/2,y radius=\E];
  \draw (-\W/2,0)--(-\W/2,\H);  \draw (\W/2,0)--(\W/2,\H);
  \draw (-\W/2,0) arc[start angle=180,end angle=360,x radius=\W/2,y radius=\E];
  \draw[dashed] (\W/2,0) arc[start angle=0,end angle=180,x radius=\W/2,y radius=\E];
  \draw (-\W/2,\H) arc[start angle=180,end angle=360,x radius=\W/2,y radius=\E];
  \draw[dashed] (\W/2,\H) arc[start angle=0,end angle=180,x radius=\W/2,y radius=\E];
  \draw[gray!55] (-\W/2,\yLvl) arc[start angle=180,end angle=360,x radius=\W/2,y radius=\E];
  \draw[gray!55,dashed] (\W/2,\yLvl) arc[start angle=0,end angle=180,x radius=\W/2,y radius=\E];
  \draw[fill=white,line join=miter] (-3.5,\H+1.0)--(-0.9,\H+1.0)--(-0.9,\yLvl+0.05)--(-0.74,\yLvl+0.05)--(-0.95,\yLvl-0.4)--(-1.16,\yLvl+0.05)--(-1.05,\yLvl+0.05)--(-1.05,\H+0.85)--(-3.5,\H+0.85)--cycle;
  \node[font=\scriptsize,align=center,anchor=south] at (-2.2,\H+1.05){fresh medium $S^*$};
  \draw[fill=white,line join=miter] (1.05,\yLvl-0.4)--(1.05,\H+0.85)--(3.1,\H+0.85)--(3.1,\H+0.68)--(3.5,\H+0.925)--(3.1,\H+1.17)--(3.1,\H+1.0)--(0.9,\H+1.0)--(0.9,\yLvl-0.4)--cycle;
  \node[font=\scriptsize,align=center,anchor=south] at (1.95,\H+1.2){washout $(S_1,S_2)$};
  \node[align=center,font=\scriptsize\itshape] at (0,1.75){well-mixed\\ bulk};
  \node[anchor=east,align=right,font=\scriptsize] at (-\W/2-0.35,0.95){biofilm of\\ thickness $h$};
  \draw[-Latex] (-\W/2-0.3,0.95)--(-\W/2+\wall+\bio/2,0.95);
  \coordinate (zc) at (\W/2-\film/2,1.9);
  \draw[thick] (zc) circle (0.42);
\end{scope}
\draw[densely dashed,gray!60] (zc) -- (-1.6,3.7);
\draw[densely dashed,gray!60] (zc) -- (-1.6,-0.5);
\begin{scope}[shift={(-1.6,-0.5)}]
  \def\Lh{3.0}\def\Ct{4.2}\def\Bw{2.0}\def\Buni{1.1}
  \fill[bulkaqua] (\Lh,0) rectangle (\Lh+\Buni,\Ct);
  \shade[left color=bulkaqua,right color=white] (\Lh+\Buni,0) rectangle (\Lh+\Bw,\Ct);
  \fill[gray!12] (0,0) rectangle (\Lh,\Ct);
  \fill[gray!55] (-0.4,0) rectangle (0,\Ct);
  \draw[line width=1.1pt] (0,0)--(0,\Ct);
  \foreach \y in {0.2,0.55,...,4.1} \draw[gray!75,line width=0.4pt] (-0.4,\y)--(0,\y-0.2);
  \draw[densely dashed] (\Lh,0)--(\Lh,\Ct);
  \draw[blue!60!black,line width=1pt] (0,1.8) .. controls (0.6364*\Lh,1.9) and (0.8788*\Lh,2.55) .. (\Lh,3.2);
  \fill[blue!60!black] (\Lh,3.2) circle(1.2pt);
  \node[blue!60!black,font=\tiny,anchor=west] at (\Lh+0.06,3.2){$S_1$};
  \node[blue!60!black,font=\tiny,anchor=south] at (0.4848*\Lh,2.3){$c_1$};
  \draw[red!65!black,line width=1pt] (0,1.6) .. controls (0.4848*\Lh,1.52) and (0.7879*\Lh,1.05) .. (\Lh,0.6);
  \fill[red!65!black] (\Lh,0.6) circle(1.2pt);
  \node[red!65!black,font=\tiny,anchor=west] at (\Lh+0.06,0.6){$S_2$};
  \node[red!65!black,font=\tiny,anchor=north] at (0.4242*\Lh,1.1){$c_2$};
  \node[font=\tiny,rotate=90,gray!45!black] at (-0.22,\Ct/2){substratum};
  \node[font=\tiny,align=center] at (\Lh/2,0.28){biofilm};
  \node[font=\tiny,align=center,blue!30!black,anchor=west] at (\Lh+0.15,\Ct-0.5){well-mixed bulk\\[-1pt]$Q,S_1,S_2$};
\end{scope}
\end{tikzpicture}%
}%
\caption{The reactor~\eqref{C} together with a magnification of its wall. \textbf{Left:} the continuously stirred
tank, fed fresh medium of propionate at the inflow concentration $S^*$ and dilution rate $D$. The tank is well
mixed, so the bulk propionate $S_1$, acetate $S_2$ and suspended biomass $Q$ depend on $t$ alone. The
dissolved $S_1,S_2$ are washed out with the effluent at rate $D$, and the biomass $Q$ is removed at rate $k_2$.
A biofilm of thickness $h$ grows on the wall. \textbf{Right:} the biofilm across its
depth $0\le z\le h$. At the wall $z=0$, no concentration flux occurs, i.e. $\partial_z c_i(0)=0$, while at the interface $z=h$, each concentration matches its bulk value, i.e. $c_i(h)=S_i$. 
}
\label{fig:reactor}
\end{figure}

The outline of this paper is as follows. In Section~\ref{Sec2}, we first recall auxiliary results for the subproblem~\eqref{C1a}–\eqref{C2a} and its rescaled formulation~\eqref{U}, which are then used to establish the global well-posedness of~\eqref{C} and needed later on. In Section~\ref{Sec3}, we provide conditions for the local and global stability of the trivial (washout) equilibrium. In the regime where this equilibrium is unstable, we prove the existence of a unique persistence equilibrium in Section~\ref{SS4.2} and establish its local stability in Section~\ref{Sec5}; the more lengthy calculations are postponed to Appendix~\ref{App}. In Section~\ref{sec:6}, we show uniform persistence of solutions, which is then used in Section~\ref{sec:7} to prove global stability of the nontrivial equilibrium under more restrictive assumptions.  Finally, Appendix~\ref{AppN} presents simulations illustrating the stability analysis of the washout and persistence equilibria.  \\

We conclude this introduction with the general assumptions that are imposed throughout this paper. We assume that
\begin{subequations}\label{GAshooting}
\begin{align}
&r\in  {\rm C}^1( \R_+^2,\R_+)\,,\qquad d\in {\rm C}^1(\R_+,\R_+)\label{r1}
\end{align}
satisfy
\begin{align}
r(0,y)=0\,,\quad \partial_2 r(x,y)< 0<\partial_1 r(x,y)\,,\quad (x,y)\in (0,\infty)^2\,,\label{r2}
\end{align}
where $\R_+=[0,\infty)$, and
\begin{align}
d'(x)\ge 0\,,\quad x\ge 0\,, \qquad  d(x)>0\,,\quad x>0\,. \label{d}
\end{align}
For the diffusion coefficients $\kappa_1$ and $\kappa_2$ of propionate and acetate, respectively, we assume that
\begin{align}\label{kappa}
\kappa_2>\kappa_1\,.
\end{align}
Finally, we always consider nonnegative initial values
\begin{equation}
\big(h^0,S_1^0,S_2^0,Q^0\big)\in\R_+^4\,.
\end{equation}
\end{subequations}
Additional assumptions will be introduced later where required.\\

The conditions~\eqref{r1}--\eqref{kappa} on $r$ and $d$  are imposed for simplicity rather than optimality. In~\cite{GHE21}, the thermodynamic growth function is chosen as a Monod rate weighted by a thermodynamic inhibition factor of the form
\begin{equation}\label{R:kinetics}
g(S_1,S_2)=\Big[\mu\,\frac{S_1}{K+S_1}\Big(1-\Gamma\,\frac{S_2}{S_1}\Big)\Big]_+  
\end{equation}
with $\mu,K,\Gamma>0$, which satisfies the strict monotonicities $\partial_2 g<0<\partial_1 g$ on its active regime $\{\Gamma S_2<S_1\}$.

\section{Well-Posedness}\label{Sec2} 

We give a brief sketch of the global well-posedness of~\eqref{C} and provide some additional properties of the solution required later on. 
It is sometimes convenient to write~\eqref{C1a}-\eqref{C2b} on the fixed interval $[0,1]$ by introducing dimensionless variables  (neglecting time-dependence for a moment)
\begin{align}\label{dvar}
y=\frac{z}{h}\quad \text{ and }\quad u_i(y)=c_i\bigl(yh\bigr)\,,
\end{align}
that is,
\begin{subequations}\label{U}
\begin{align}
\kappa_1 \partial_{y}^2u_1&= h^2r(u_1,u_2)\,, \quad 0<y<1\,,\quad \label{BBB2}\\
\partial_{y}u_1(0)&=0\,,\quad u_1(1)=S_1\,, \label{BBB4}\\
\kappa_2 \partial_{y}^2u_2&= -h^2r(u_1,u_2)\,, \quad 0<y<1\,,\quad \label{BBB2x}\\
\partial_{y}u_2(0)&=0\,,\quad u_2(1)=S_2\,. \label{BBB4x}
\end{align}
\end{subequations}
Note that the subproblem~\eqref{C3}-\eqref{C7} now becomes
\begin{subequations}\label{UU}
\begin{align}
h'(t)&=k_1h(t)\int_0^{1}r\big(u_1(t,y),u_2(t,y)\big)\, \rd y-bh(t) +\frac{\alpha}{\beta} Q-d\big(h(t)\big)h(t)\,, \label{C3x}\\
S_1(t)'&=D\big(S^*-S_1(t)\big)-r\big(S_1(t),S_2(t)\big)Q(t)-\beta h(t)\int_0^{1} r\big(u_1(t,y),u_2(t,y)\big) \, \rd y \,, \label{C4ax}\\
S_2(t)'&=-DS_2(t)+r\big(S_1(t),S_2(t)\big)Q(t)+\beta h(t) \int_0^{1} r\big(u_1(t,y),u_2(t,y)\big)\, \rd y \,, \label{C4bx}\\
Q'(t)&= \big(k_1r\big(S_1(t),S_2(t)\big)-k_2-\alpha\big)Q(t)+\beta d\big(h(t)\big)h(t)\,, \label{C5x}
\end{align}
for $t>0$ subject to the nonnegative initial conditions
\begin{align}\label{C7x}
h(0)&=h^0\,,\quad S_1(0)=S_1^0\,,\quad S_2(0)=S_2^0\,,\quad Q(0)=Q^0 \,.
\end{align}
\end{subequations}
As for the subproblem~\eqref{U} we note:

\begin{prop}\label{P1}
Assume \eqref{GAshooting}. There exists  a mapping
$$
\Bigl[ (h,S_1,S_2)\mapsto u[h,S_1,S_2]\Bigr]\in {\rm C}^1\big(\R_+^3,
{\rm C}^2([0,1],\R^2)\big)
$$
such that $$u=(u_1,u_2)=u[h,S_1,S_2]\in {\rm C}^2([0,1],\R^2)$$ is the unique solution to the boundary value problem~\eqref{U}.
For $(h,S_1,S_2)\in\R_+^3$ it holds that
\begin{subequations}\label{ccc}
\begin{align}
&0\le u_1(y)\le S_1\,,\quad S_2\le u_2(y)\le \frac{\kappa_1}{\kappa_2}S_1+S_2\,,\qquad y\in [0,1]\,,\label{ccc1}\\
& \kappa_1 u_1(y)+ \kappa_2 u_2(y)= \kappa_1 S_1+\kappa_2 S_2\,,\quad  y\in [0,1]\,. \label{ccc3}
\end{align}
\end{subequations}
Moreover, if $h>0$, then
\begin{align}\label{G}
\frac{\partial}{\partial {S_1}}\int_0^1 r\big(u[h,S_1,S_2](y)\big) \rd y \ge  \frac{\partial}{\partial {S_2}}\int_0^1 r\big(u[h,S_1,S_2](y)\big) \rd y
\end{align}
and
\begin{align}\label{GR}
\frac{\partial}{\partial h}\int_0^1 r\big(u[h,S_1,S_2](y)\big)\,\rd y\le 0\,.
\end{align}
\end{prop}

\begin{proof}
Except for~\eqref{G} and~\eqref{GR},  this can be obtained as in~\cite[Proposition~2.1]{GW_QAM} or~\cite[Proposition~2.1]{NW26}, and we thus omit details. 

As for~\eqref{G} we infer from~\eqref{U} and~\eqref{ccc3} that
$w:=\partial_{S_1}u_1-\partial_{S_2}u_1$ satisfies
\begin{subequations}\label{p1}
\begin{align}
\kappa_1 \partial_{y}^2w&= h^2\big(p(y)w+q(y)\big)\,, \quad 0<y<1\,,\\
\partial_{y}w(0)&=0\,,\quad w(1)=1\,,  
\end{align}
\end{subequations}
with
$$
p(y):=\partial_1 r(u(y))-\frac{\kappa_1}{\kappa_2}\partial_2 r(u(y))> 0\,,\qquad q(y):=\left(\frac{\kappa_1}{\kappa_2}-1\right)\partial_2 r(u(y))\ge 0
$$
for $y\in [0,1]$ due to~\eqref{r2} and \eqref{kappa}. Hence $\partial_{y}^2w(1)>0$ and it then readily follows from~\eqref{p1} by a contradiction argument that
 $\partial_y w(1)\ge 0$. 
This implies that
$$
h^2\frac{\partial}{\partial {S_1}}\int_0^1 r(u(y)) \rd y- h^2\frac{\partial}{\partial {S_2}}\int_0^1 r(u(y)) \rd y =\kappa_1\partial_y w(1)\ge 0
$$
and \eqref{G} follows.

As for~\eqref{GR} we set $r_j(y):=\partial_j r(u(y))$ for $y\in [0,1]$ and $j=1,2$. From ~\eqref{U} we deduce that
\begin{align*}
\kappa_1 \partial_y^2\big(\partial_h u_1\big)(y)&= 2hr(u(y))+h^2\big(r_1(y) \partial_h u_1(y)+r_2(y) \partial_h u_2(y)\big)\,, \quad 0<y<1\,, \\
\kappa_2 \partial_y^2\big(\partial_h u_2\big)(y)&= -2h r(u(y))-h^2\big(r_1(y) \partial_h u_1(y)+r_2(y) \partial_h u_2(y)\big)\,, \quad 0<y<1\,, 
\end{align*}
subject to
$$
\partial_y\partial_h u_i(0)=0\,,\qquad \partial_h u_i(1)=0
$$ 
for $i=1,2$ while~\eqref{ccc3} implies
\begin{align}\label{vv}
\kappa_1\partial_h u_1(y)+\kappa_2\partial_h u_2(y)
=0 \, , \quad y \in [0,1]\,. 
\end{align}
Consequently, $v:=\partial_h u_1\in {\rm C}^2([0,1])$ satisfies
\begin{align*}
\kappa_1 \partial_y^2v(y)&=h^2\Big(r_1(y)-\tfrac{\kappa_1}{\kappa_2}r_2(y)\Big) v(y)
+2h r(u(y))\,,\qquad y\in(0,1)\,,\\
\partial_{y}v(0)&=0\,,\quad v(1)=0\,,
\end{align*}
with 
\begin{equation}\label{r12x}
r_1(y)-\tfrac{\kappa_1}{\kappa_2}r_2(y)>0\,,\quad y\in (0,1)\,,
\end{equation} 
owing to~\eqref{r2}. Testing with $v_+=\max\{v,0\}$  and using
 $v_+(1)=0$ yields
\begin{align*}
0&\le \frac{h^2}{\kappa_1}\int_0^1\chi_{[v>0]}v^2(y) \Big(r_1(y)-\tfrac{\kappa_1}{\kappa_2}r_2(y)\Big)\,\rd y+\frac{2h}{\kappa_1}\int_0^1 v_+(y) r(u(y))\,\rd y\\
&=-\int_0^1\vert \partial_y v_+(y)\vert^2\,\rd y
\end{align*} 
and thus $v_+\equiv 0$. That is, $v=\partial_h u_1\le 0$ in $[0,1]$. Consequently, we obtain from~\eqref{vv} and~\eqref{r12x} that
\begin{align*}
\frac{\partial}{\partial h}\int_0^1 r\big(u(y)\big)\,\rd y& =\int_0^1\big(r_1(y) \partial_h u_1(y)+r_2(y) \partial_h u_2(y)\big)\,\rd y\\
&=\int_0^1\Big(r_1(y)-\tfrac{\kappa_1}{\kappa_2}r_2(y)\Big)\partial_h u_1(y) \rd y\le 0
\end{align*} 
hence~\eqref{GR}.
\end{proof}

Clearly, uniqueness implies that 
\begin{align}\label{I2}
u[0,S_1,S_2]= (S_1,S_2)\,,\qquad u[h,0,S_2]=(0,S_2)
\,.
\end{align}

Based on the previous proposition we obtain the global well-posedness of \eqref{U}-\eqref{UU}:

\begin{thm}\label{T1}
Assume \eqref{GAshooting}. Then, given $(h^0, S_1^0,S_2^0, Q^0)\in \R_+^4$, there is a
unique global solution 
$$
(h,S_1,S_2,Q,u)\in {\rm C}^1\big([0,\infty),\R_+^4\times {\rm C}^2([0,1],\R^2)\big)
$$
to \eqref{U}-\eqref{UU}.
 Moreover, $(h,S_1,S_2,Q)$ is bounded in $\R_+^4$ and it holds that
\begin{subequations}\label{I1}
\begin{align}
&0\le u_1(t,y)\le S_1(t)\,,\quad 0\le S_2(t)\le u_2(t,y)\le \frac{\kappa_1}{\kappa_2}S_1(t)+S_2(t)\,,\qquad y\in [0,1]\,,\label{ccc1x}\\
& \kappa_1 u_1(t,y)+ \kappa_2 u_2(t,y)= \kappa_1 S_1(t)+\kappa_2 S_2(t)\,,\quad  y\in [0,1]\,, \label{ccc3x}
\end{align}
and
\begin{align}\label{29c}
S_1(t)+S_2(t)=e^{-Dt}(S_1^0+S_2^0)+S^*\big(1-e^{-Dt}\big)
\end{align}
\end{subequations}
for $t\ge 0$.  In fact, the global semiflow generated by~\eqref{UU} in $\R_+^4$ admits a compact global attractor~$\mathcal{A}$ consisting of complete orbits bounded on $\R$.
\end{thm}

\begin{proof} 
We use Proposition~\ref{P1} in order to write~\eqref{U}-\eqref{UU} as a Cauchy problem in $\R_+^4$ of the form 
\begin{equation}\label{CP}
X'=f(X)\,,\quad t \ge 0\,,\qquad X(0)=X^0\,,
\end{equation} 
for $X=(h,S_1,S_2,Q)$ with $f\in{\rm C}^1(\R_+^4)$ and take into account that $f_i(X)\ge 0$ for $X\in \R_+^4$ with $X_i=0$ according to~\eqref{r2} and \eqref{I2}. Thus, for each $X^0=(h^0, S_1^0,S_2^0, Q^0)\in \R_+^4$ there exists a unique maximal nonnegative solution \mbox{$X=(h,S_1,S_2,Q)$} to~\eqref{CP} (and thus to~\eqref{U}-\eqref{UU}) satisfying~\eqref{I1} on the maximal interval of existence $[0,T_m)$. In fact, since
\begin{align*}
\frac{\rd}{\rd t}\big(\beta h+\mu S_1+S_2+Q\big)(t)&= D\big[\mu S^*-\mu S_1(t)-S_2(t)\big]\\
&\quad  +
\left[\big(k_1 +1-\mu\big) r\big(S_1(t),S_2(t)\big)-k_2\right]Q(t)\\
&\quad 
+\left[(k_1+1-\mu)\int_0^1 r\bigl(u(t,y)\bigr)\,\rd y-b\right]\beta h(t)
\end{align*}
for $t\in [0,T_m)$ and any $\mu\ge 0$, we deduce for $W:=\beta h+(k_1+1)S_1+S_2+Q$ that 
\begin{align}\label{W}
W'(t)\le D(k_1+1)S^*- qW(t)\,,\quad t\in [0,T_m)\,,
\end{align}
with $q:=\min\{D,k_2,b\}>0$. Consequently, $W$ and (thus the solution $(h,S_1,S_2,Q)$) is bounded, hence $T_m=\infty$. In fact,
\begin{equation}\label{Wbound}
\limsup_{t\to\infty}W(t)\le \frac{D(k_1+1)S^*}{q}
\end{equation}
so that  
every orbit eventually enters and remains in a bounded sublevel set of~$W$, i.e. the semiflow generated by~\eqref{UU} in $\R_+^4$ is point dissipative and asymptotically smooth. Therefore, it admits a compact global attractor $\mathcal{A}$ by~\cite[Theorem~3.4.6]{Hale1988} consisting of complete orbits bounded on $\R$.

\end{proof}

In particular, Theorem~\ref{T1} implies that all solution orbits are relatively compact, and therefore their $\omega$-limit sets are nonempty. The long-term dynamics are investigated in the sequel.

\section{The Washout Equilibrium}\label{Sec3}

Clearly,~\eqref{CP} admits the washout equilibrium
$$
h=0\,,\quad S_1=S^*\,,\quad S_2=0\,,\quad  Q=0\,.
$$
We now linearize~\eqref{CP} at this trivial steady state.
Recalling that $u[0,S^*,0]\equiv (S^*,0)$, the Jacobian of $f$ at this equilibrium is
\begin{equation}\label{Jac}
\partial f(0,S^*,0,0)=\left(\begin{matrix}
k_1 r(S^*,0)-b-d(0) & 0 & 0 & \alpha/\beta\\[4pt]
-\beta r(S^*,0) & -D & 0& -r(S^*,0)\\[4pt]
\beta r(S^*,0) & 0 & -D& r(S^*,0)\\[4pt]
\beta d(0) & 0 & 0& k_1r(S^*,0)-k_2-\alpha
\end{matrix}\right)\,.
\end{equation}
The eigenvalues of this matrix can be computed explicitly, yielding a double eigenvalue
$$
\lambda_0:=-D<0
$$
and
\begin{equation}\label{lambda}
\lambda_\pm:=\frac{1}{2}\left(p+q \pm\sqrt{ (p-q)^2+4\alpha d(0)}\right)
\end{equation}
with 
$$
p:=k_1r(S^*,0)-b-d(0) \,,\qquad q:=k_1r(S^*,0)-k_2-\alpha\,.
$$
Since $\lambda_\pm\in\R$ with $\lambda_+\geq \lambda_-$, the local stability of the equilibrium $(0,S^*,0,0)$ is determined solely by the sign of $\lambda_+$.  
In fact, we have $\lambda_+<0$
and thus local stability provided that
\begin{equation}\label{stab}
-pq+\alpha d(0)< 0 \quad \text{ and }\quad p+q<0\,,
\end{equation}
while
$\lambda_+ >0$
and thus local instability provided that
\begin{equation}\label{instab}
-pq+\alpha d(0)>0 \quad \text{ or }\quad p+q>0\,.
\end{equation}
Consequently, as in~\cite{GHE21}, we obtain the following local and global stability result for the washout equilibrium:

\begin{thm}\label{P2}
Assume~\eqref{GAshooting}. Then, the  washout equilibrium $(h,S_1,S_2,Q)=(0,S^*,0,0)$ is unstable provided that
$$
d(0)\big( k_1r(S^*,0)-k_2\big) > \big(k_1r(S^*,0)-b\big)\big(k_1r(S^*,0)-k_2-\alpha\big)
$$
or
$$ 
2 k_1r(S^*,0)-k_2-\alpha > d(0)+b \,,
$$
while it is locally asymptotically stable  provided that both conditions
\begin{align}\label{S1}
d(0)\big( k_1r(S^*,0)-k_2\big)<\big(k_1r(S^*,0)-b\big)\big(k_1r(S^*,0)-k_2-\alpha\big) 
\end{align}
and
\begin{align}\label{S2}
2 k_1r(S^*,0)-k_2-\alpha<d(0)+b  
\end{align}
hold. In fact,  it is globally asymptotically stable  provided that~\eqref{S1} holds and~\eqref{S2} is strengthened to $k_1r(S^*,0)<k_2$.
\end{thm}

\begin{proof}
In view of the previous discussion, it remains only to prove the global stability assuming~\eqref{S1} and $k_1r(S^*,0)<k_2$. To this end, let~$(h^0, S_1^0,S_2^0, Q^0)\in [0,\infty)^4$ be an arbitrary initial value and choose $\delta>0$ small enough such that  
\begin{align}\label{S1cx}
d(0)\big( k_1r(S^*+\delta,0)-k_2\big)<\big(k_1r(S^*+\delta,0)-b\big)\big(k_1r(S^*+\delta,0)-k_2-\alpha\big) 
\end{align}
and $k_1r(S^*+\delta,0)<k_2$.
Now, since~\eqref{S1cx} is equivalent to
$$
0\le d(0)\alpha< \big(d(0)-k_1r(S^*+\delta,0)+b\big)\big(\alpha+k_2-k_1r(S^*+\delta,0)\big)\,,
$$
it follows from $k_1r(S^*+\delta,0)<k_2$ that we may choose  $\eta>0$ satisfying
$$
\max\left\{\beta,\frac{\beta d(0)}{d(0)-k_1r(S^*+\delta,0)+b}\right\}<\eta<
\frac{\beta \big(\alpha+k_2-k_1r(S^*+\delta,0)\big)}{\alpha}\,.
$$
Next, \eqref{29c} entails that $S_1'(t)\le D(S^*-S_1(t))$ for $t\ge 0$ so that there exists $T_\delta>0$ such that $S_1(t)\le S^*+\delta$ for $t\ge T_\delta$.  Recalling~\eqref{ccc1x} and the monotonicity properties of $r$ from~\eqref{r2}, we deduce that
$$
r\big(u_1(t,y),u_2(t,y)\big)\le r\big(S_1(t),S_2(t)\big)\le r(S^*+\delta,0)\,,\quad y\in [0,1]\,,\quad t\ge T_\delta\,.
$$
Therefore, for $t\ge T_\delta$ we have that
\begin{align*}
\frac{\rd}{\rd t}\big(\eta h+Q\big)(t)&= \eta k_1h(t)\int_0^{1}r\big(u_1(t,y),u_2(t,y)\big)\, \rd y-\eta bh(t)\\
&\quad +\big(\beta-\eta\big) d\big(h(t)\big)h(t)
+\left(\frac{\eta \alpha}{\beta}+k_1r\big(S_1(t),S_2(t)\big)-k_2-\alpha\right)Q(t)\\
&\le -\left[ -k_1r(S^*+\delta,0)+b+\left(1-\frac{\beta}{\eta}\right)d(0)\right] \eta h(t)\\
&\quad -\left[-k_1 r(S^*+\delta,0)+k_2+\alpha-\frac{\alpha\eta}{\beta}\right] Q(t)\\
&\le -m\big(\eta h+Q\big)(t)
\end{align*}
with
$$
m:=\min\left\{ -k_1r(S^*+\delta,0)+b+\left(1-\frac{\beta}{\eta}\right) d(0)\,,\, -k_1 r(S^*+\delta,0)+k_2+\alpha-\frac{\alpha\eta}{\beta}\right\}>0\,.
$$
Consequently,
$$
\lim_{t\to\infty} \big(\eta h+Q\big)(t)=0\,.
$$
It then readily follows from this and~\eqref{C4bx} and~\eqref{29c} that
$$
\lim_{t\to\infty} \big(S_1(t),S_2(t)\big)=(S^*,0)\,.
$$
That is, the trivial equilibrium is globally stable.
\end{proof}

Note that the proof of Theorem~\ref{P2} yields the explicit strict Lyapunov functional $\eta h+Q$ in case that~\eqref{S1}  and  $k_1r(S^*,0)<k_2$ hold.

\section{The Persistence Equilibrium}\label{SS4.2} 

In Theorem~\ref{P2}, we provided conditions for the (global) stability of the trivial washout equilibrium. We now establish the existence and uniqueness of a nontrivial steady state in the complementary case where these conditions fail, thereby confirming the conjecture made in~\cite{GHE21}. Since an explicit representation is unlikely to be available, we instead employ a shooting argument, which simultaneously yields existence and uniqueness.\\

A stationary solution to~\eqref{C} necessarily satisfies 
\begin{subequations}\label{Stat}
\begin{align}
S_1+S_2=S^*\,,\qquad \kappa_1 c_1(z)+\kappa_2 c_2(z)=(\kappa_1-\kappa_2)S_1 +\kappa_2 S^*\,,\quad z\in [0,h]\,,
\end{align}
so that $S_2$ and $c_2$ are uniquely determined by $S_1$ and $c_1$. Consequently, writing 
\begin{align}
S:=S_1\,,\qquad c:=c_1\,,\qquad {\rm r}(S):=r(S,S^*-S)
\end{align} 
\end{subequations}
for notational simplicity, the stationary problem for~\eqref{C} is equivalent to finding numbers $h,S,Q>0$ and a function $c=c(z)$, $z\in [0,h]$, such that
\begin{subequations}\label{A}
\begin{align}
&\kappa_1\partial_z^2 c(z)\stackrel{!}{=}r\left(c(z),\kappa_2^{-1}\big((\kappa_1-\kappa_2)S+\kappa_2 S^*-\kappa_1 c(z)\big)\right)\,, \quad 0<z<h\,, \label{A1}\\
&\partial_z c(0)\stackrel{!}{=}0\,,\quad
 c(h)\stackrel{!}{=}S\,,\quad
k_1\kappa_1 \partial_z c(h)\stackrel{!}{=}d(h)h+bh-\frac{\alpha}{\beta}Q\,, \label{A4}
\end{align}
where $Q$ and $S$ are determined by the conditions
\begin{equation}\label{A5x}
Q\stackrel{!}{=}\frac{\beta d(h)h}{\alpha+k_2-k_1{\rm r}(S)}
\end{equation}
and
\begin{equation*}
DS+{\rm r}(S)Q+\beta\kappa_1 \partial_zc(h)\stackrel{!}{=}D S^*\,,
\end{equation*}
respectively. Using~\eqref{A4} and~\eqref{A5x} to express $\partial_zc(h)$, the latter equation equivalently reads as
\begin{equation}\label{A6}
k_1DS+\frac{\beta k_2}{\alpha+k_2-k_1{\rm r}(S)} d(h)h+b\beta h\stackrel{!}{=}k_1D S^*\,.
\end{equation}
\end{subequations}
Clearly,~\eqref{A5x} enforces that 
\begin{equation}\label{Delta}
\Delta(S):=\alpha+k_2-k_1{\rm r}(S)>0\,.
\end{equation}
In fact, we shall impose that
\begin{subequations}\label{ZZ}
\begin{align}\label{Z1x}
&k_2\ge k_1 r(S^*,0) 
\end{align}
and the uniform instability condition
\begin{align}\label{h2eff}
\Big(k_1\,r\big(S^*,\kappa_2^{-1}(\kappa_2-\kappa_1)S^*\big)-b\Big)\big(\alpha+k_2-k_1 r(S^*,0)\big)>\big(k_2-k_1 r(S^*,0)\big)d(0)\,.
\end{align}
\end{subequations}
 Since ${\rm r}$ is non-decreasing according to~\eqref{r2} so that ${\rm r}(S)\le{\rm r}(S^*)=r(S^*,0)$ for $S\in[0,S^*]$, condition~\eqref{Z1x} ensures that
$$
\Delta(S)\ge\Delta(S^*)\ge\alpha>0\,,\quad S\in[0,S^*]\,.
$$
Moreover, since $ r\big(S^*,\kappa_2^{-1}(\kappa_2-\kappa_1)S^*\big)\le r(S^*,0)$ due to~\eqref{r2} and \eqref{kappa}, condition~\eqref{h2eff} implies that
\begin{equation*}
\big(k_1 r(S^*,0)-b\big)\big(\alpha+k_2-k_1 r(S^*,0)\big)>\big(k_2-k_1 r(S^*,0)\big)d(0)\,,
\end{equation*}
which, in turn, implies the instability of the trivial steady state by Theorem~\ref{P2}. Under these assumptions, we prove the existence of a unique nontrivial equilibrium:

\begin{thm}\label{T5}
Assume~\eqref{GAshooting} and ~\eqref{ZZ}. Then, there exists a unique persistence equilibrium
\begin{align*}
&\big(h_\star,S_{\star,1},S_{\star,2},Q_\star\big) \in (0,\infty)\times (0,S^*)\times (0,S^*)\times (0,\infty)\,,\\
& \big(c_{\star,1},c_{\star,2}\big)\in {\rm C}^2([0,h_\star],\R_+^2)
\end{align*}
of~\eqref{C}. It is defined by the unique solution
$(h_\star,S_{\star,1},Q_\star, c_{\star,1})$
to~\eqref{A} and by~\eqref{Stat}.
\end{thm}

In order to prove Theorem~\ref{T5}, we first show that, for each $h>0$, there exists a unique solution $S=S(h)$ of~\eqref{A6}. The remaining subproblem~\eqref{A1}--\eqref{A5x} is then treated by a shooting argument analogous to that in~\cite{NW26}. However, an additional difficulty arises: after solving~\eqref{A6} and expressing $S$ as a function of $h$, the right-hand side of the resulting shooting problem exhibits an additional dependence on $h$. As we shall see, this requires a more careful analysis.
\\

We begin by solving~\eqref{A6} for $S$, thereby expressing $S=S(h)$ as a function of $h$.

\begin{lem}\label{L20}
 Assume~\eqref{GAshooting} and \eqref{Z1x}.
There exist a unique $h_*>0$ and a unique, strictly decreasing function $S\in {\rm C}^1([0,h_*], [0,S^*])$ with $S(0)=S^*$ and $S(h_*)=0$ such that 
\begin{equation}\label{A6y}
k_1DS(h)+\frac{\beta k_2}{\Delta(S(h))} d(h)h+b\beta h=k_1D S^*\,,\qquad h\in [0,h_*]\,.
\end{equation}
\end{lem}

\begin{proof}
Setting
$$
F(h,S):=k_1D(S-S^*)+\frac{\beta k_2}{\Delta(S)} d(h)h+b\beta h\,,\qquad h\ge 0\,,\quad S\in [0,S^*]\,,
$$
we solve $F(h,S)=0$ for $S=S(h)$.
 Noticing that $\Delta(0)=\alpha+k_2$ since ${\rm r}(0)=0$ according to~\eqref{r2}, we have
$$
F(h,0)=-k_1D S^*+\frac{\beta k_2}{\alpha+k_2} d(h)h+b\beta h\,,\quad h>0\,.
$$
Therefore, since $d$ is non-decreasing, there exists a unique $h_*>0$ such that 
\begin{align}\label{c12}
F(h,0)<0\,,\quad 0\le h< h_*\,,\qquad F(h_*,0)=0\,.
\end{align} 
Moreover, 
\begin{align}\label{c13}
F(h,S^*)= \frac{\beta k_2}{\Delta(S^*)} d(h)h+b\beta h>0\,,\quad 0< h\le h_*\,.
\end{align}
Next, observe that
\begin{align}\label{c14}
\partial_SF(h,S)=k_1D+ \frac{\beta k_2}{\Delta(S)^2} k_1{\rm r}'(S)d(h)h>0\,,\qquad h\ge 0\,,\quad S\in [0,S^*]\,,
\end{align}
since ${\rm r}'(S)\ge 0$ due to~\eqref{r2}. Therefore, \eqref{c12}-\eqref{c14} ensure that for each $h\in [0,h_*]$ there is a unique $S(h)\in [0,S^*]$ such that $F(h,S(h))=0$. 
Clearly, we have $S(0)=S^*$ and $S(h_*)=0$. Moreover, the Implicit Function Theorem and \eqref{c14} imply that $S\in C^1([0,h_*])$. Finally, noticing from~\eqref{d} that
\begin{align*}
\partial_h F(h,S)&= \frac{\beta k_2}{\Delta(S)} \big(d'(h)h+d(h)\big)+b\beta >0\,,\qquad h\in [0,h_*]\,,\quad S\in [0,S^*]\,,
\end{align*}
we may differentiate $F(h,S(h))=0$ with respect to $h\in [0,h_*]$ and use~\eqref{c14} to get
$$
S'(h)=-\frac{\partial_h F(h,S(h))}{\partial_SF(h,S(h))}<0\,,\quad h\in [0,h_*]\,.
$$
This proves Lemma~\ref{L20}.
\end{proof}

Since $S=S(h)$ and  $Q=Q(h)$ are now expressed in terms of $h$, we are left with determining  $c=c(z)$ and $h\in (0,h_*]$ uniquely such that
\begin{subequations}\label{AA}
\begin{align}
&\kappa_1\partial_z^2 c(z)\stackrel{!}{=}r\left(c(z)\,,\,\kappa_2^{-1}\big((\kappa_1-\kappa_2)S(h)+\kappa_2 S^*-\kappa_1 c(z)\big)\right)\,, \quad 0<z<h\,, \label{A1q}\\
&\partial_z c(0)\stackrel{!}{=}0\,,\quad
 c(h)\stackrel{!}{=}S(h)\,,\quad
k_1\kappa_1 \partial_z c(h)\stackrel{!}{=}bh+\frac{k_2-k_1 {\rm r}(S(h))}{\Delta(S(h))}d(h)h\,. \label{AA4}
\end{align}
\end{subequations}

To this end, we use a shooting argument based on the following auxiliary lemma:

\begin{lem}\label{L17} 
 Assume~\eqref{GAshooting} and \eqref{Z1x}.
Set
$$
R({\rm c},h_0):=r\left({\rm c},\kappa_2^{-1}\big((\kappa_1-\kappa_2)S(h_0)+\kappa_2 S^*-\kappa_1 {\rm c}\big)\right)\,,\quad {\rm c}\ge 0\,,\quad h_0\in [0,h_*]\,.
$$
Then, for each $(\mu,h_0)\in [0,S^*]\times [0,h_*]$, there exists a unique non-negative function $${\rm c}(\cdot,\mu,h_0)\in {\rm C}^2([0,\infty))$$ satisfying
\begin{subequations}\label{shooting}
\begin{align}
&\kappa_1\partial_z^2 {\rm c}(z)=R({\rm c}(z),h_0)\,, \quad z>0\,, \label{A1ss}\\
&{\rm c}(0)=\mu\,,\quad \partial_z {\rm c}(0)=0\,. \label{A4ss}
\end{align}
\end{subequations}
The functions $z\mapsto {\rm c}(z,\mu,h_0)$ and $z\mapsto \partial_z{\rm c}(z,\mu,h_0)$ increase strictly for $\mu\in (0,S^*]$, while ${\rm c}(\cdot,0,h_0)\equiv 0$. Moreover,
$$
\big[(\mu,h_0)\mapsto {\rm c}(\cdot,\mu,h_0)\big]\in {\rm C}^1\big( [0,S^*]\times [0,h_*],{\rm C}^2([0,T])\big)
$$
for each $T>0$ 
with 
\begin{equation}\label{ww}
\partial_\mu {\rm c}(z,\mu,h_0)\ge 1\,,\qquad \partial_z\partial_\mu {\rm c}(z,\mu,h_0)\ge 0
\end{equation} 
for $z\ge 0$ and  $(\mu,h_0)\in [0,S^*]\times [0,h_*]$. In addition, 
$$
\partial_{h_0} {\rm c}(z,\mu,h_0)<0\,,\quad \partial_z\partial_{h_0} {\rm c}(z,\mu,h_0)<0
$$
for $ z> 0$ and $(\mu,h_0)\in (0,S^*]\times [0,h_*]$, while $\partial_{h_0}{\rm c}(z,0,h_0)=\partial_z\partial_{h_0}{\rm c}(z,0,h_0)=0$ for $z\ge 0$ and $h_0\in [0,h_*]$.
\end{lem}

\begin{proof}
 Clearly, the function $R\in {\rm C}^1\big(\R_+\times [0,h_*]\big)$ satisfies
\begin{equation}\label{dR}
\partial_{\rm c}R({\rm c},h_0)>0\,,\quad {\rm c}\ge 0\,,\quad h_0\in[0,h_*]\,,
\end{equation}
according to~\eqref{r2},~\eqref{kappa} and Lemma~\ref{L20}. The existence of a unique function ${\rm c}$ with the stated regularity properties thus follows along the lines of~\cite[Lemma~3.7]{GW_QAM} (extending $r$ to a function in ${\rm C}^1(\R^2,\R)$ retaining~\eqref{r2}). 

Since $R(0,h_0)=0$ due to~\eqref{r2}, we have for $\mu=0$ that ${\rm c}(\cdot,0,h_0)\equiv 0$ by uniqueness. For $\mu>0$ we obtain from~\eqref{A1ss} and~\eqref{r2} that $\partial_z^2{\rm c}(z,\mu,h_0)>0$, so that ${\rm c}(\cdot,\mu,h_0)$ and $\partial_z{\rm c}(\cdot,\mu,h_0)$ are strictly increasing.

For fixed $(\mu,h_0)\in [0,S^*]\times [0,h_*]$, the function
$w:=\partial_\mu {\rm c}(\cdot,\mu,h_0)\in{\rm C}^2\bigl([0,\infty)\bigr)$
satisfies
\begin{subequations}
\begin{align}
\kappa_1\partial_z^2 w(z)&=\partial_{\rm c} R\big({\rm c}(z),h_0\big)w(z)\,, \quad z>0\,, \label{w1}\\
w(0)&=1\,,\quad \partial_z w(0)=0\,. 
\end{align}
\end{subequations}
Then $w$ is (at least initially) convex due to~\eqref{dR} and $w(0)=1$. It follows that $\partial_zw(z)\geq 0$ and
\mbox{$w(z)\ge w(0)=1$} for $z\geq 0$ small and then for all $z\geq 0$ by
monotonicity. 

Similarly, for fixed $(\mu,h_0)\in (0,S^*]\times [0,h_*]$, the function
$v:=\partial_{h_0} {\rm c}(\cdot,\mu,h_0)\in{\rm C}^2\bigl([0,\infty)\bigr)$
satisfies
\begin{subequations}
\begin{align}
\kappa_1\partial_z^2 v(z)&=\partial_{\rm c} R({\rm c}(z),h_0)v(z)+q(z)\,, \quad z>0\,,\label{v1} \\
v(0)&=0\,,\quad \partial_z v(0)=0\,,\label{v2}
\end{align}
\end{subequations}
with
\begin{align*}
\partial_{\rm c} R({\rm c}(z),h_0)> 0
\end{align*}
and
\begin{align*}
 q(z):=\partial_2 r\left({\rm c}(z)\,,\,\kappa_2^{-1}\big((\kappa_1-\kappa_2) S(h_0)+\kappa_2 S^*-\kappa_1 {\rm c}(z)\big)\right)\frac{\kappa_1-\kappa_2}{\kappa_2}S'(h_0)<0
\end{align*}
for $z\ge 0$ according to~\eqref{r2}, ~\eqref{kappa}, and Lemma~\ref{L20}. Therefore, the function $v$ is concave	and hence $v(z)<0$ and $\partial_zv(z)<0$ for $z> 0$. This yields the assertion.
\end{proof}

It remains to show that there are unique $\mu$ and $h$ such that ${\rm c}(\cdot,\mu,h)$ solves~\eqref{AA}. Using  monotonicity arguments, we establish that the boundary conditions from~\eqref{AA4} can be satisfied. We first focus on the Dirichlet boundary condition.

\begin{lem}\label{L21}
Assume~\eqref{GAshooting} and \eqref{Z1x}.
Let ${\rm c}$  be the solution to the shooting problem~\eqref{shooting}.
Then, there exists a unique function $h\in {\rm C}^1([0,S^*]\times [0,h_*])$   such that 
\begin{align}\label{id}
{\rm c}\big(h(\mu,h_0),\mu,h_0\big)=S\big(h(\mu,h_0)\big)\,,\quad (\mu,h_0)\in [0,S^*]\times [0,h_*]\,.
\end{align}
For fixed $h_0\in [0,h_*]$, the mapping $\mu\mapsto h(\mu,h_0)$ is strictly decreasing on $[0,S^*]$ 
with $h(0,h_0)=h_*$ and $h(S^*,h_0)=0$, while for fixed $\mu\in (0,S^*)$, the mapping $h_0\mapsto h(\mu,h_0)$ is strictly increasing on~$[0,h_*]$.
\end{lem}

\begin{proof}
Setting
$$
A(z,\mu,h_0):={\rm c}(z,\mu,h_0)-S(z)\,,\qquad z\in [0,h_*]\,,\quad (\mu,h_0)\in [0,S^*]\times [0,h_*]\,,
$$ 
we infer from Lemma~\ref{L17} and Lemma~\ref{L20} that
\begin{align*}
A(0,\mu,h_0)&={\rm c}(0,\mu,h_0)-S(0)=\mu-S^*\le 0\,,\\
A(h_*,\mu,h_0)&={\rm c}(h_*,\mu,h_0)-S(h_*)={\rm c}(h_*,\mu,h_0)\ge 0\,,
\end{align*}
with equality only for $\mu=S^*$ and $\mu=0$, respectively, while
\[
\partial_z A(z,\mu,h_0)=\partial_z{\rm c}(z,\mu,h_0)-S'(z)>0\,,\quad z\in [0,h_*]\,,
\]
since $S'<0$ on $[0,h_*]$ by Lemma~\ref{L20} and $\partial_z{\rm c}\ge0$ by Lemma~\ref{L17} with
$\partial_z{\rm c}(0,\mu,h_0)=0$, so that $A(\cdot,\mu,h_0)$ is strictly increasing on $[0,h_*]$. Hence, for
every $(\mu,h_0)\in [0,S^*]\times [0,h_*]$ there is a unique $h(\mu,h_0)\in [0,h_*]$ with
\[
A(h(\mu,h_0),\mu,h_0)=0\,.
\]
Since ${\rm c}(\cdot,0,h_0)\equiv 0$ implies $A(h_*,0,h_0)=0$ while ${\rm c}(0,S^*,h_0)=S^*=S(0)$  implies
$A(0,S^*,h_0)=0$, we have $h(0,h_0)=h_*$ and $h(S^*,h_0)=0$. 
This establishes~\eqref{id}. The Implicit Function Theorem then ensures $h\in {\rm C}^1([0,S^*]\times [0,h_*])$. Differentiating~\eqref{id} and using  Lemma~\ref{L17} yields for $(\mu,h_0)\in [0,S^*]\times [0,h_*]$
\begin{align}\label{381}
\big[\partial_z{\rm c}(h(\mu,h_0),\mu,h_0)-S'(h(\mu,h_0)) \big]\partial_{\mu} h(\mu,h_0)=-\partial_\mu {\rm c}(h(\mu,h_0),\mu,h_0)\le -1
\end{align}
while for $(\mu,h_0)\in (0,S^*)\times [0,h_*]$ we have
\begin{align}\label{382}
\big[\partial_z {\rm c}(h(\mu,h_0),\mu,h_0)-S'(h(\mu,h_0)) \big]\partial_{h_0} h(\mu,h_0)=-\partial_{h_0} {\rm c}(h(\mu,h_0),\mu,h_0)>0\,.
\end{align}
Therefore, since
\begin{equation}\label{poss}
\partial_z{\rm c}(h(\mu,h_0),\mu,h_0)-S'(h(\mu,h_0))>0
\end{equation}
by Lemma~\ref{L20} and Lemma~\ref{L17}, we deduce for $h_0\in [0,h_*]$ that
$\partial_{\mu} h(\mu,h_0)<0$ for $\mu\in [0,S^*]$ and $\partial_{h_0} h(\mu,h_0)>0$ for $\mu\in (0,S^*)$ as claimed.
\end{proof}

In order to show that the Neumann boundary condition in~\eqref{AA4} can be satisfied, we require the following auxiliary result:

\begin{lem}\label{L22}
Assume~\eqref{GAshooting} and~\eqref{ZZ}.  Given $(z,\mu,h_0)\in [0,h_*]\times [0,S^*]\times [0,h_*]$ and $h\in [0,h_*]$ set
\[
w(z,\mu,h_0):=\partial_\mu{\rm c}(z,\mu,h_0)\,,\qquad
\psi(h):=\frac{k_2-k_1{\rm r}(S(h))}{\Delta(S(h))}\,,
\]
and
\begin{equation}\label{Mdef}
\begin{aligned}
M(z,\mu,h_0):=&-\big(k_1 R({\rm c}(z,\mu,h_0),h_0)-b\big)w(z,\mu,h_0)\\
&+k_1\kappa_1\big(\partial_z{\rm c}(z,\mu,h_0)-S'(h(\mu,h_0))\big)\partial_z w(z,\mu,h_0)\\
&+\psi(h(\mu,h_0))d(h(\mu,h_0))\,.
\end{aligned}
\end{equation}
Then, for fixed $(\mu,h_0)\in [0,S^*]\times [0,h_*]$, the function
$$
M(\cdot,\mu,h_0):[0,h_*]\to \R
$$ 
is strictly increasing. Moreover, the function
$M(0,\cdot,h_0):[0,S^*]\to\R$ with
\begin{equation}\label{M0}
M(0,\mu,h_0)=- k_1 R(\mu,h_0)+b +\frac{k_2-k_1{\rm r}\big(S(h(\mu,h_0))\big)}{\Delta\big(S(h(\mu,h_0))\big)}d\big(h(\mu,h_0)\big)\,,\quad \mu\in [0,S^*]\,,
\end{equation}
is strictly decreasing for each $h_0\in [0,h_*]$, and there exists (a unique) $\underline{\mu}(h_0)\in (0,S^*)$ such that 
$$
M(0,\underline{\mu}(h_0),h_0)=0\,.
$$ 
In particular,
\begin{equation}\label{M}
M(z,\mu,h_0)> M(0,\mu,h_0)\ge 0\,,\qquad z\in (0,h_*]\,,\quad \mu\in [0,\underline{\mu}(h_0)]\,,\quad h_0\in [0,h_*] \,.
\end{equation}
\end{lem}

\begin{proof}
Differentiating~\eqref{Mdef} with respect to $z$ and using~\eqref{A1ss} and~\eqref{w1} yields
\begin{equation}\label{dzM}
\partial_z M(z,\mu,h_0)=b\partial_z w(z,\mu,h_0)-k_1\partial_{\rm c} R({\rm c}(z,\mu,h_0),h_0)w(z,\mu,h_0)S'(h(\mu,h_0))>0
\end{equation}
for $(z,\mu,h_0)\in [0,h_*]\times [0,S^*]\times [0,h_*]$, where we used the monotonicity properties of $R$, $w$, and $S$ (see~\eqref{dR}, \eqref{ww}, and Lemma~\ref{L20}). This ensures that \mbox{$M(\cdot,\mu,h_0):[0,h_*]\to \R$} is strictly increasing.

Next, the boundary conditions ${\rm c}(0,\mu,h_0)=\mu$, $\partial_z{\rm c}(0,\mu,h_0)=0$, $w(0,\mu,h_0)=1$, and $\partial_z w(0,\mu,h_0)=0$ from Lemma~\ref{L17} reduce~\eqref{Mdef} at $z=0$ to~\eqref{M0}.
Differentiating~\eqref{M0} then with respect to $\mu$  yields
\begin{equation}\label{dmuM}
\partial_\mu M(0,\mu,h_0)=-k_1\partial_{\rm c} R(\mu,h_0)+(\psi d)'(h(\mu,h_0))\partial_\mu h(\mu,h_0)<0\,.
\end{equation}
Indeed, the first term in~\eqref{dmuM} is strictly negative by~\eqref{dR}, while the second is non-positive since $\partial_\mu h(\mu,h_0)<0$ by~\eqref{381} and
\[
(\psi d)'(h)=\psi'(h)d(h)+\psi(h)d'(h)\ge 0\,,
\]
where $\psi'\geq 0$ because $\rm r$ is non-decreasing and $S$ is strictly decreasing, $\psi\ge 0$ since $k_1{\rm r}(S(h))\le k_1 r(S^*,0)\le k_2$, and $d'\ge 0$  by~\eqref{d}. Consequently, $M(0,\cdot,h_0):[0,S^*]\to \R$ is strictly decreasing.

Since $h(0,h_0)=h_*$, $S(h_*)=0$, and
$R(0,h_0)=0$
by~\eqref{r2}, it follows from~\eqref{M0} that
\[
M(0,0,h_0)=b+\frac{k_2}{\alpha+k_2}\,d(h_*)>0\,,
\]
while $h(S^*,h_0)=0$ and $S(0)=S^*$ together with~\eqref{M0} that
\begin{align*}
M(0,S^*,h_0)&=-\big(k_1 R(S^*,h_0)-b\big)+\frac{k_2-k_1{\rm r}\big(S^*\big)}{\Delta(S^*)}\,d(0)\\
&\le -\big(k_1 r\big(S^*,\kappa_2^{-1}(\kappa_2-\kappa_1)S^*\big)-b\big)+\frac{k_2-k_1{\rm r}\big(S^*\big)}{\Delta(S^*)}\,d(0)<0\,,
\end{align*}
where we used for the second inequality assumption~\eqref{h2eff} and for the first inequality the fact that
$$
R(S^*,h_0)=r\big(S^*,\kappa_2^{-1}(\kappa_2-\kappa_1)(S^*-S(h_0))\big)\ge r\big(S^*,\kappa_2^{-1}(\kappa_2-\kappa_1)S^*\big)
$$
by the monotonicity properties~\eqref{r2} of $r$ and $S(h_0)\in [0,S^*]$.

Finally, since $M(0,\cdot,h_0)$ is strictly decreasing on $[0,S^*]$, with $M(0,0,h_0)>0$ and $M(0,S^*,h_0)<0$, there exists a unique $\underline\mu(h_0)\in(0,S^*)$ such that $M(0,\underline\mu(h_0),h_0)=0$. The just established monotonicity properties of $M$ now imply~\eqref{M}.
\end{proof}

In fact, the mapping $\underline{\mu}\in {\rm C}^1\big([0,h_*], (0,S^*)\big)$ is strictly increasing as is readily seen by the Implicit Function Theorem and differentiating the identity $M(0,\underline\mu(h_0),h_0)=0$.\\

We are now in a position to address the Neumann boundary condition in~\eqref{AA4}. We thus examine the zeros $\mu$ of the function $B$ given by
\begin{equation}\label{400}
\begin{split}
B(\mu,h_0):=\, & k_1\kappa_1 \partial_z {\rm c}(h(\mu,h_0),\mu,h_0)-bh(\mu,h_0)\\
&\ -\frac{k_2-k_1 {\rm r}(S(h(\mu,h_0)))}{\Delta(S(h(\mu,h_0)))}d\big(h(\mu,h_0)\big)h(\mu,h_0)
\end{split}
\end{equation}
for $\mu\in [0,S^*]$ and $h_0\in [0,h_*]$. Note that $B(S^*,h_0)=0$ for any $h_0\in [0,h_*]$, since $h(S^*,h_0)=0$. We next show that $B(\cdot,h_0)$ possesses exactly one zero in the interval $[0,S^*]$ other than  $\mu=S^*$  (which then corresponds to the unique nontrivial equilibrium while $\mu=S^*$ gives the trivial equilibrium).

\begin{lem}\label{L40}
Assume~\eqref{GAshooting} and~\eqref{ZZ}. Let  $\underline{\mu}\in {\rm C}^1\big([0,h_*], (0,S^*)\big)$ be as in Lemma~\ref{L22}. Then, there exists a function $\mu_*\in {\rm C}^1([0,h_*])$ such that, for every
$h_0\in[0,h_*]$, one has $\mu_*(h_0)\in(0,\underline{\mu}(h_0)]$, and
$\mu=\mu_*(h_0)$ and $\mu=S^*$ are the only zeros of
$B(\cdot,h_0):[0,S^*]\to\R$.
\end{lem}

\begin{proof}
Fix $h_0\in [0,h_*]$. It follows from \eqref{r2}, Lemma~\ref{L20}, Lemma~\ref{L17}, and Lemma~\ref{L21} that
$$
h(0,h_0)=h_*\,,\quad S(h_*)=0\,,\quad {\rm r}(0)=r(0,S^*)=0\,, \quad \partial_z {\rm c}(\cdot,0,h_0)\equiv 0\,,
$$ 
hence
$$
B(0,h_0)=-bh_*-\frac{k_2}{\alpha+k_2}d(h_*)h_*<0\,.
$$
while clearly $B(S^*,h_0)=0$, since $h(S^*,h_0)=0$.  

Next, observe from Lemma~\ref{L17} that ${\rm c}(z,\mu,h_0)\ge{\rm c}(0,\mu,h_0)=\mu$ for $z\ge 0$ and $\mu\in[0,S^*]$, hence, by~\eqref{A1ss} and the monotonicity~\eqref{dR} of $R$,  we obtain
\[
\kappa_1\partial_z{\rm c}(h(\mu,h_0),\mu,h_0)=\int_0^{h(\mu,h_0)}R({\rm c}(z,\mu,h_0),h_0)\,\rd z\ge R(\mu,h_0)h(\mu,h_0)\,,\quad \mu\in[0,S^*]\,.
\]
Therefore, 
$$
B(\mu,h_0)\ge -h(\mu,h_0)M(0,\mu,h_0)>0\,,\qquad \mu\in(\underline\mu(h_0),S^*)\,,
$$
according to Lemma~\ref{L22} and since $h(\mu,h_0)>0$. 

Finally, we show that $B(\cdot,h_0)$ is strictly increasing on the interval $[0,\underline\mu(h_0))$.  Setting $w(z,\mu,h_0):=\partial_\mu {\rm c}(z,\mu,h_0)$ and using~\eqref{381} and \eqref{shooting} yields that
\begin{align*}
\big[\partial_z{\rm c}&(h(\mu,h_0),\mu, h_0)-S'(h(\mu,h_0))\big] \partial_\mu B(\mu,h_0)\\
&=-\big[ k_1 R\big({\rm c}(h(\mu,h_0),\mu,h_0),h_0\big)-b\big] w\big(h(\mu, h_0),\mu,h_0\big)\\
&\quad +k_1\kappa_1 \big[\partial_z{\rm c}(h(\mu,h_0),\mu,h_0)-S'(h(\mu,h_0))\big]\partial_zw\big(h(\mu,h_0),\mu,h_0\big) \\
&\quad +\frac{k_2-k_1{\rm r}\big(S(h(\mu,h_0))\big)}{\Delta\big(S(h(\mu,h_0))\big)}\big[d'\big(h(\mu,h_0)\big)h(\mu,h_0)+d\big(h(\mu,h_0)\big)\big]w\big(h(\mu,h_0),\mu,h_0\big)\\
&\quad  -\frac{k_1\alpha}{\big(\Delta\big(S(h(\mu,h_0))\big)\big)^2} {\rm r}'\big(S(h(\mu,h_0))\big)S'(h(\mu,h_0)) d\big(h(\mu,h_0)\big)h(\mu,h_0)w\big(h(\mu, h_0),\mu, h_0\big)\\
&\ge -\big[ k_1 R\big({\rm c}(h(\mu,h_0),\mu,h_0),h_0\big)-b\big] w\big(h(\mu, h_0),\mu,h_0\big)\\
&\quad +k_1\kappa_1 \big[\partial_z{\rm c}(h(\mu,h_0),\mu,h_0)-S'(h(\mu,h_0))\big]\partial_zw\big(h(\mu,h_0),\mu,h_0\big) \\
&\quad +\frac{k_2-k_1{\rm r}\big(S(h(\mu,h_0))\big)}{\Delta\big(S(h(\mu,h_0))\big)}d\big(h(\mu,h_0)\big)\\
&= M\big(h(\mu,h_0),\mu,h_0\big)
\end{align*}
for $\mu\in [0,S^*]$, where we used for the inequality the positivity of the involved quantities, $w(z,\mu,h_0)\ge 1$, that ${\rm r}$ and $d$ are increasing, and that $S$ is decreasing along with \eqref{Z1x} to neglect the last term and part of the second-to-last term. Therefore, since $h(\mu,h_0)>0$ for $\mu\in [0,S^*)$ according to Lemma~\ref{L21} and since $\underline{\mu}(h_0)\in (0,S^*)$ by Lemma~\ref{L22}, 
we derive from~\eqref{poss} and~\eqref{M} that 
\begin{equation}\label{Bq}
\partial_\mu B(\mu,h_0)>0\,,\quad \mu\in [0,\underline{\mu}(h_0)]\,.
\end{equation}  
Consequently, the function $B(\cdot,h_0)$ is strictly increasing on $[0,\underline\mu(h_0)]$, with $B(0,h_0)<0$ and $B(\mu,h_0)>0$ for  $\mu\in (\underline\mu(h_0),S^*)$. Therefore, there exists a unique $\mu_*(h_0)\in(0,\underline\mu(h_0)]$ such that
$$
B(\mu_*(h_0),h_0)=0\,.
$$
The Implicit Function Theorem and~\eqref{Bq} imply that $\mu_*\in {\rm C}^1([0,h_*])$.
\end{proof}

The final step is to show the existence of a unique $h_0\in(0,h_*)$ with $h(\mu_*(h_0),h_0)=h_0$. We establish this through the following auxiliary result and the lemma that builds on it.

\begin{lem}\label{L30}
Assume~\eqref{GAshooting} and \eqref{Z1x}.  Define
\begin{equation*}
K(z,\mu,h_0):=k_1\kappa_1\big[\partial_z\partial_{h_0}{\rm c}(z,\mu,h_0)\partial_\mu{\rm c}(z,\mu,h_0)-\partial_z\partial_\mu{\rm c}(z,\mu,h_0)\partial_{h_0}{\rm c}(z,\mu,h_0)\big]
\end{equation*}
for $(z,\mu,h_0)\in[0,\infty)\times[0,S^*]\times[0,h_*]$. Then $K(z,\mu,h_0)\le 0$.
\end{lem}

\begin{proof}
It readily follows from~\eqref{v2} that
$K(0,\mu,h_0)=0$ while
differentiating $K$ with respect to $z$ and using~\eqref{w1} and \eqref{v1} gives
\begin{align}\label{Kprime}
\partial_z K(z,\mu,h_0)=&k_1\partial_\mu{\rm c}(z,\mu,h_0)\frac{\kappa_1-\kappa_2}{\kappa_2}S'(h_0)\\
&\times\partial_2 r\Big({\rm c}(z,\mu,h_0),\kappa_2^{-1}\big((\kappa_1-\kappa_2)S(h_0)+\kappa_2 S^*-\kappa_1{\rm c}(z,\mu,h_0)\big)\Big)\le 0 \nonumber
\end{align}
for $z\ge 0$, where we use Lemma~\ref{L20}, Lemma~\ref{L17},~\eqref{kappa}, and~\eqref{r2} to conclude the inequality.
\end{proof}

\begin{lem}\label{P30}
Assume~\eqref{GAshooting} and~\eqref{ZZ}. Let $h=h(\mu,h_0)$ and $\mu_*=\mu_*(h_0)$ be the functions from Lemma~\ref{L21} and Lemma~\ref{L40}, respectively. Then, there exists a unique $h_\star\in (0,h_*)$ such that $h\big(\mu_*(h_\star),h_\star)=h_\star$.
\end{lem}

\begin{proof}
Set
$$
D(h_0):=h\big(\mu_*(h_0),h_0)-h_0\,,\quad h_0\in [0,h_*]\,,
$$
and note from Lemma~\ref{L21} that
$$
D(0)=h(\mu_*(0),0)>0\,,\qquad  D(h_*)=h(\mu_*(h_*),h_*)-h_*<0\,.
$$
Therefore, the assertion follows, provided that we can show that $D'(h_0)<0$ for $h_0\in[0,h_*]$, where
\begin{equation}\label{D}
D'(h_0)=\partial_\mu h(\mu_*(h_0),h_0)\mu_*'(h_0)+\partial_{h_0}h(\mu_*(h_0),h_0)-1\,.
\end{equation}
Differentiating $B(\mu_*(h_0),h_0)=0$ with respect to $h_0$ yields
\begin{equation}\label{D1}
\partial_\mu B(\mu_*(h_0),h_0)\mu_*'(h_0)=-\partial_{h_0}B(\mu_*(h_0),h_0)\,.
\end{equation}
Writing $B$ from~\eqref{400} in the form 
$$
B(\mu,h_0)=k_1\kappa_1\partial_z{\rm c}(h(\mu,h_0),\mu,h_0)-\phi(h(\mu,h_0)) 
$$
with
$$
\phi(h):=bh+\frac{k_2-k_1{\rm r}(S(h))}{\Delta(S(h))}d(h)h\,,
$$
it follows that
\begin{equation}\label{D2}
\begin{split}
\partial_\mu B(\mu,h_0)&=k_1 R\big({\rm c}(h(\mu,h_0),\mu,h_0),h_0\big) \partial_\mu h(\mu,h_0)\\
&\quad +
k_1\kappa_1 \partial_\mu\partial_z {\rm c}(h(\mu,h_0),\mu,h_0) - \phi'\big(h(\mu,h_0)\big)\partial_\mu h(\mu,h_0)
\end{split}
\end{equation}
and
\begin{equation}\label{D3}
\begin{split}
\partial_{h_0} B(\mu,h_0)&=k_1 R\big({\rm c}(h(\mu,h_0),\mu,h_0),h_0\big) \partial_{h_0} h(\mu,h_0)\\
&\quad +
k_1\kappa_1 \partial_{h_0}\partial_z {\rm c}(h(\mu,h_0),\mu,h_0) - \phi'\big(h(\mu,h_0)\big)\partial_{h_0} h(\mu,h_0)\,.
\end{split}
\end{equation}
Now, we infer from~\eqref{D} and~\eqref{D1} that
\begin{align*}
\partial_{\mu}B(\mu_*(h_0),h_0)\big(D'(h_0)+1\big)&=-\partial_{h_0}B(\mu_*(h_0),h_0) \partial_{\mu}h\big(\mu_*(h_0),h_0)\\
&\quad +\partial_{\mu}B(\mu_*(h_0),h_0) \partial_{h_0} h(\mu_*(h_0),h_0) 
\end{align*}
and thus, taking into account~\eqref{D2} and~\eqref{D3}, we obtain
\begin{align*}
\partial_{\mu}B(\mu_*(h_0),h_0)\big(D'(h_0)+1\big)&=-k_1\kappa_1 \partial_{h_0}\partial_z {\rm c}\big(h(\mu_*(h_0),h_0),\mu_*(h_0),h_0\big) \partial_{\mu} h\big(\mu_*(h_0),h_0\big)\\
&\quad +k_1\kappa_1 \partial_{\mu}\partial_z {\rm c}\big(h(\mu_*(h_0),h_0),\mu_*(h_0),h_0\big) \partial_{h_0} h\big(\mu_*(h_0),h_0\big)\,.
\end{align*}
Therefore, invoking~\eqref{381}-\eqref{382} we derive
\begin{align*}
\big[\partial_z{\rm c}&(h(\mu_*(h_0),h_0),\mu_*(h_0),h_0)-S'(h(\mu_*(h_0),h_0)) \big]\partial_{\mu}B(\mu_*(h_0),h_0)\big(D'(h_0)+1\big)\\
&=k_1\kappa_1 \partial_{h_0}\partial_z {\rm c}\big(h(\mu_*(h_0),h_0),\mu_*(h_0),h_0\big) \partial_{\mu} {\rm c}\big(h(\mu_*(h_0),h_0),\mu_*(h_0),h_0\big)\\
&\quad -k_1\kappa_1 \partial_{\mu}\partial_z {\rm c}\big(h(\mu_*(h_0),h_0),\mu_*(h_0),h_0\big) \partial_{h_0} {\rm c}\big(h(\mu_*(h_0),h_0),\mu_*(h_0),h_0\big)\\
&=K\big(h(\mu_*(h_0),h_0),\mu_*(h_0),h_0\big)\,.
\end{align*}
Consequently, since $K$ is negative according to Lemma~\ref{L30}, we conclude from~\eqref{poss} and~\eqref{Bq} that $D'(h_0)\le -1$ for $h_0\in [0,h_*]$.
\end{proof}

Gathering the previous observations, we are in a position to prove Theorem~\ref{T5}.

\begin{proof}[Proof of Theorem~\ref{T5}]
Let $h_\star\in(0,h_*)$ be the value determined in Lemma~\ref{P30} such that $h(\mu_*(h_\star),h_\star)=h_\star$, and set
$$
S_{\star,1}:=S(h_\star)\in(0,S^*)\,,\quad Q_\star:=\frac{\beta d(h_\star) h_\star}{\Delta(S_{\star,1})} > 0\,, \quad c_{\star,1}:={\rm c}(\cdot,\mu_*(h_\star),h_\star)\in {\rm C}^2([0,h_\star])\,,
$$
with $S=S(h)$ from Lemma~\ref{L20} and ${\rm c}={\rm c}(\cdot,\mu,h_0)$ from Lemma~\ref{L17}. Lemma~\ref{L20}, Lemma~\ref{L17}, Lemma~\ref{L40}, and Lemma~\ref{P30} then imply that $(h_\star,S_{\star,1},Q_\star,c_{\star,1})$ satisfies~\eqref{A}, and by construction this is the only nontrivial solution. Consequently, setting
$$
S_{\star,2}:=S^*-S_{\star,1}\,,\qquad \kappa_2c_{\star,2}(z):=(\kappa_1-\kappa_2)S_{\star,1}+\kappa_2 S^*-\kappa_1\,c_{\star,1}(z)\,,\quad z\in[0,h_\star]\,,
$$
the tuple $(h_\star,S_{\star,1},S_{\star,2},Q_\star,c_{\star,1},c_{\star,2})$ is the unique nontrivial equilibrium of~\eqref{C}.
\end{proof}

In the sequel, we focus on the stability properties of the persistence equilibrium established in this section.

\section{Local Stability of the Persistence Equilibrium}\label{Sec5} 

Under a slightly stronger assumption than in Theorem~\ref{T5}, we prove the local stability of the persistence equilibrium using the linearization principle:

\begin{thm}\label{T6}
Assume~\eqref{GAshooting}, \eqref{ZZ}, and let
\begin{equation}\label{d1}
\min\{2\alpha,k_2\}\ge k_1r(S^*,0)\,\,.
\end{equation}
Then the persistence equilibrium $(h_\star,S_{\star,1},S_{\star,2},Q_\star)$ provided by Theorem~\ref{T5} is locally asymptotically stable.
\end{thm}


\begin{proof}
We argue by the principle of linearized stability. Recall that Theorem~\ref{T5} entails
$$
S_{\star}=(S_{\star,1},S_{\star,2})=(S_{\star,1},S^*-S_{\star,1})\in (0,S^*)^2\,,
$$
and the monotonicity of $r$ together with~\eqref{d1} yields
\begin{equation}\label{Delta1}
\Delta:=\alpha+k_2-k_1r(S_{\star,1},S_{\star,2})>\alpha+k_2-k_1r(S^*,0)\ge\alpha>0\,.
\end{equation}
We write $r(S_{\star})=r(S_{\star,1},S_{\star,2})$ for brevity. The equilibrium relation~\eqref{C5} gives
\begin{equation}\label{Q}
Q_\star=\frac{\beta d(h_\star)h_\star}{\Delta}\,,
\end{equation}
and dividing~\eqref{C3} at the equilibrium by $h_\star$ and inserting~\eqref{Q} yields
\begin{equation}\label{b1}
 k_1\int_0^1 r(u_\star(y)) \rd y-b=d(h_\star)-\frac{\alpha d(h_\star)}{\Delta}\,.
\end{equation}
As in~\eqref{CP}, consider~\eqref{UU} in the form $X'=f(X)$ for $X=(h,S_1,S_2,Q)\in\R^4$. Recalling from Proposition~\ref{P1} that the mapping $[h,S]\mapsto u[h,S]$ is of class $C^1$, setting $u_\star=u[h_\star,S_\star]$, and using~\eqref{b1}, we compute the Jacobian $A:=\partial f( X_\star)$ of $f$ at  $X_\star=(h_\star,S_{\star,1},S_{\star,2},Q_\star)$  as
{\small\begin{equation*} 
\left(\begin{matrix}
   k_1 h_\star   H -\frac{\alpha  d(h_\star)}{\Delta}-d'(h_\star)  h_\star  &  k_1  h_\star   G_1 &  k_1  h_\star   G_2 & \frac{\alpha}{\beta}\\[4pt]
-\beta R-\beta h_\star H  &-D-r_1(S_\star)Q_\star-\beta h_\star  G_1  &-r_2(S_\star)Q_\star-\beta h_\star  G_2  & -r(S_\star)\\[4pt]
\beta R+\beta h_\star H  & r_1(S_\star)Q_\star+\beta h_\star  G_1 & -D+r_2(S_\star)Q_\star+\beta h_\star  G_2 & r(S_\star)\\[4pt]
\beta d(h_\star)+\beta d'(h_\star)h_\star  & k_1 r_1(S_\star) Q_\star &  k_1 r_2(S_\star) Q_\star & -\Delta
\end{matrix}\right)\,,
\end{equation*}}
where we set $r_j:=\partial_j r$ and
\begin{subequations}\label{Z}
\begin{equation}\label{GH}
H:=\frac{\partial}{\partial h} \int_0^1 r(u_\star(y)) \rd y= \int_0^1 \Big\{r_1(u_\star(y))\partial_h u_{\star,1}(y)+r_2(u_\star(y))\partial_h u_{\star,2}(y)\Big\} \rd y\le 0\,,
\end{equation}
\begin{equation*}
 G_i:=\frac{\partial}{\partial {S_i}}\int_0^1 r(u_\star(y)) \rd y= \int_0^1 \Big\{r_1(u_\star(y))\partial_{S_i} u_{\star,1}(y)+r_2(u_\star(y))\partial_{S_i} u_{\star,2}(y)\Big\} \rd y
\end{equation*}
with
\begin{equation}\label{GH1}
G_{12}:=G_1-G_2\ge 0
\end{equation}
and
\begin{equation}\label{GH3}
R:=\int_0^1 r(u_\star(y)) \rd y\in [0,r(S_\star)]
\end{equation}
according to Proposition~\ref{P1}. Since adding multiples of a row or column to another does not alter the determinant of a matrix, it readily follows that the eigenvalues of $A$ coincide with the eigenvalues of the matrix 
{\small\begin{equation*}
\left(\begin{matrix}
 k_1 h_\star   H -\frac{\alpha  d(h_\star)}{\Delta}-d'(h_\star)  h_\star  &  k_1  h_\star   G_1 & - k_1  h_\star   G_{12} & \frac{\alpha}{\beta}\\[4pt]
0  &-D &0 & 0\\[4pt]
\beta R+\beta h_\star H  & r_1(S_\star)Q_\star+\beta h_\star  G_1 & -D-r_{12}(S_\star) Q_\star-\beta h_\star  G_{12} &\quad r(S_\star)\\[4pt]
\beta d(h_\star)+\beta d'(h_\star)h_\star  & k_1 r_1(S_\star) Q_\star &  -k_1 r_{12}(S_\star) Q_\star & -\Delta
\end{matrix}\right)\,,
\end{equation*}}
where (see~\eqref{r2})
\begin{equation}\label{GH4}
r_{12}(S_\star):=r_{1}(S_\star)-r_{2}(S_\star)> 0\,.
\end{equation}
\end{subequations}
Finally, using Laplace expansion along the second row entails that the eigenvalues of the linearization are given by $-D<0$ and the eigenvalues of the $(3\times 3)$-matrix 
\begin{equation}\label{MA}
A':=\left(\begin{matrix}
  k_1 h_\star   H -\frac{\alpha  d(h_\star)}{\Delta}-d'(h_\star)  h_\star   & - k_1  h_\star   G_{12} & \frac{\alpha}{\beta}\\[4pt]
\beta R+\beta h_\star H   & -D-r_{12}(S_\star) Q_\star-\beta h_\star  G_{12} &\quad r(S_\star)\\[4pt]
\beta d(h_\star)+\beta d'(h_\star)h_\star   &  -k_1 r_{12}(S_\star) Q_\star & -\Delta
\end{matrix}\right)\,.
\end{equation}
It thus remains to show that all eigenvalues of $A'$ have negative real parts under assumptions~\eqref{GAshooting}, \eqref{ZZ}, and \eqref{d1}. To this end, we apply the Routh–Hurwitz criterion. The technical details are postponed to Appendix~\ref{App}. The principle of linearized stability the implies that the nontrivial equilibrium $(h_\star,S_{\star,1},S_{\star,2},Q_\star)$ is locally asymptotically stable.
\end{proof}

The next objective is to prove the global stability of the persistence equilibrium. For this purpose, we first establish the uniform persistence of solutions in the next section.

\section{Uniform Persistence}\label{sec:6}

In this section, we show that each component of a solution $(h,S_1,S_2,Q)$ persists when the nontrivial equilibrium exists and whenever $(h^0,Q^0)\not=(0,0)$. More precisely:

\begin{thm}\label{persistence}
Assume~\eqref{GAshooting} and~\eqref{ZZ}. Then there exist $\underline{h}, \underline{S}_1, \underline{S}_2,\underline{Q}>0$, independent of the initial value, such that any solution $(h,S_1,S_2,Q)$ to~\eqref{C} with $(h^0,Q^0)\not=(0,0)$ satisfies
\begin{equation}\label{persistX}
\liminf_{t\to\infty} \big(h(t),S_1(t),S_2(t),Q(t)\big)\ge \big(\underline{h},\underline{S}_1,\underline{S}_2,\underline{Q}\big)\,.
\end{equation}
\end{thm}

 In the following, we assume~\eqref{GAshooting} and~\eqref{ZZ} and consider the solution $(h,S_1,S_2,Q)$ to~\eqref{C}  corresponding to an initial value in $\R_+^4$ with $(h^0,Q^0)\not=(0,0)$. We divide the proof of Theorem~\ref{persistence} into two parts. First, we establish the uniform persistence of the biofilm, proving that the thickness $h$ and the biomass $Q$ remain bounded away from zero.

\begin{lem}\label{Tpersist}
There exist $\underline{h}>0$ and $\underline{Q}>0$, independent of the initial value, such that any solution to~\eqref{C} with $(h^0,Q^0)\not=(0,0)$ satisfies
\begin{equation}\label{persist}
\liminf_{t\to\infty} h(t)\ge\underline{h}\qquad\text{and}\qquad \liminf_{t\to\infty} Q(t)\ge\underline{Q}\,.
\end{equation}
\end{lem}

\begin{proof}
Let $\mathcal{A}$ be the 
compact attractor from Theorem~\ref{T1}.

{\bf (i)} We consider first the case of $h^0>0$ and argue by uniform persistence theory from \cite[Section~8.3]{SmithThieme2011}  (see also~\cite{HaleWaltman1989})  with persistence function $$\rho:\R_+^4\to \R\,,\quad (h,S_1,S_2,Q)\mapsto h\,.$$
Note that any complete orbit bounded on $\R$ and contained in $\R_+^4\cap\{h=0\}$ satisfies $h\equiv h'\equiv0$ and hence $Q\equiv0$ by~\eqref{C3x}, so that~\eqref{C4ax} and~\eqref{C4bx} reduce to
\begin{equation*}
S_1'=D(S^*-S_1)\,,\qquad S_2'=-DS_2\,,
\end{equation*}
whose only solutions bounded on $\R$ are $S_1\equiv S^*$ and $S_2\equiv0$. Thus, $$\mathcal{X}_0:=\{(h^0,S_1^0,S_2^0,Q^0)\in\R_+^4\,;\, h(t)=0 \text{ for all } t\ge 0\}$$
contains no homoclinic orbit and $\mathcal{A}\cap \mathcal{X}_0=\{(0,S^*,0,0)\}$, since the compact global attractor $\mathcal{A}$ consists of complete orbits bounded on $\R$. This implies that $(0,S^*,0,0)$ is isolated and acyclic (see \cite[Definition~8.14]{SmithThieme2011}).

Next, we infer from~\eqref{C3x} that 
$$
h'(t)\ge-\big(b+d(h(t))\big)h(t)\,,\quad t\ge 0\,,
$$ 
so that the set $\{h>0\}$ is forward invariant. We now claim that there exists $\delta>0$ such that every solution emanating from $\{h>0\}$ satisfies
\[
\limsup_{t\to\infty}\big|(h,S_1,S_2,Q)(t)-(0,S^*,0,0)\big|\ge\delta\,,
\]
so that the washout equilibrium is a uniform weak $\rho$-repeller. 
To this end, observe first that $\lambda_+>0$ in~\eqref{lambda} due to~\eqref{instab} and~\eqref{ZZ}, while~\eqref{Delta} and~\eqref{Z1x} entail that 
$$
\Delta(S^*)=\alpha+k_2-k_1r(S^*,0)\ge\alpha>0\,.
$$ 
Thus, whenever $h(0)=h^0>0$, we have
\[
\frac{\beta(\lambda_++\Delta(S^*))}{\alpha}h(t)+Q(t)>0\,,\quad t\ge0\,,
\]
and it follows from~\eqref{C3x},~\eqref{C5x}, and~\eqref{lambda} that
\begin{align*}
\frac{\rd}{\rd t}&\bigg(\frac{\beta(\lambda_++\Delta(S^*))}{\alpha}h+Q\bigg)(t)\\
&=\lambda_+\bigg(\frac{\beta(\lambda_++\Delta(S^*))}{\alpha}h+Q\bigg)(t)\\
&\quad +\left[k_1\left(\int_0^1 r\big(u_1(t,y),u_2(t,y)\big)\,\rd y-r(S^*,0)\right)-\big(d(h(t))-d(0)\big)\right]\frac{\beta(\lambda_++\Delta(S^*))}{\alpha}h(t)\\
&\quad+k_1\big(r(S_1(t),S_2(t))-r(S^*,0)\big)Q(t)+\beta\big(d(h(t))-d(0)\big)h(t)
\end{align*}
for $t\ge 0$.
Assume now for contradiction that the claim were false. Then, for every $\delta>0$ we find an initial value $h^0>0$ and $T_\delta>0$ such that the corresponding solution satisfies
\[
\big|(h,S_1,S_2,Q)(t)-(0,S^*,0,0)\big|\le \delta\,,\quad t\ge T_\delta\,.
\]
Since $u=(u_1,u_2)$ depends continuously on $(h,S_1,S_2)$ according to Proposition~\ref{P1} and $u[0,S^*,0]\equiv(S^*,0)$ by~\eqref{I2}, we may in fact choose $\delta>0$ small enough such that
\begin{align*}
\left\vert k_1\left(\int_0^1 r\big(u_1(t,y),u_2(t,y)\big)\,\rd y-r(S^*,0)\right)-\big(d(h(t))-d(0)\big)\right\vert\le \frac{\lambda_+}{2}\,,\quad t\ge T_\delta\,,
\end{align*}
and
\begin{align*}
\left\vert k_1\big(r(S_1(t),S_2(t))-r(S^*,0)\big)\right\vert\le \frac{\lambda_+}{2}\,,\quad t\ge T_\delta\,.
\end{align*}
Therefore, together with~\eqref{d}, we infer that
\begin{equation*}
\frac{\rd}{\rd t}\bigg(\frac{\beta(\lambda_++\Delta(S^*))}{\alpha}h+Q\bigg)(t)\ge\frac{\lambda_+}{2}\bigg(\frac{\beta(\lambda_++\Delta(S^*))}{\alpha}h+Q\bigg)(t)\,,\quad t\ge T_\delta\,.
\end{equation*}
Integration of this inequality yields a contradiction to the boundedness of the orbit.   Consequently, the washout equilibrium is indeed a uniform weak $\rho$-repeller.  
As it is also isolated and acyclic, we infer from \cite[Theorem~8.17]{SmithThieme2011} that the semiflow is uniformly weakly $\rho$-persistent.
In fact, since the set $\{h>0\}$ is forward invariant as observed above, it now readily follows from \cite[Theorem~4.5]{SmithThieme2011} that the semiflow is uniformly $\rho$-persistent.  That is, there exists $\underline{h}>0$, such that for every initial value $(h^0,S_1^0,S_2^0,Q^0)$ in $\R_+^4$ with  $h^0>0$, the corresponding solution  
satisfies
\[
\liminf_{t\to\infty}  h(t)\ge\underline{h}\,,
\]
which is the first bound in~\eqref{persist}.

Finally, given $\varepsilon\in(0,\underline{h})$ there is $T_\varepsilon>0$ such that $h(t)\ge\underline{h}-\varepsilon$ for $t\ge T_\varepsilon$. Since the map $x\mapsto d(x)x$ is nondecreasing by~\eqref{d}, we infer from~\eqref{C5x} that 
\begin{align*}
Q'(t)&\ge -(k_2+\alpha)Q(t)+\beta d(\underline{h}-\varepsilon)(\underline{h}-\varepsilon)\,,
\end{align*}
for $t\ge T_\varepsilon$. Therefore,
\begin{equation*}
\liminf_{t\to\infty} Q(t)\ge\frac{\beta d(\underline{h}-\varepsilon)(\underline{h}-\varepsilon)}{k_2+\alpha}>0\,,
\end{equation*}
where we use~\eqref{d} for the strict positivity.

{\bf(ii)} The case $Q^0>0$ follows in exactly the same way as above by interchanging the role of $h$ and $Q$ and noting from~\eqref{Wbound} that $\limsup_{t\to\infty}d(h(t))\le c_0$ for some universal constant $c_0>0$ independent of the initial value.
\end{proof}

Finally, we show that neither substrate $S_1$ nor substrate $S_2$ is depleted.

\begin{lem}
There are constants $\underline{S}_1,\underline{S}_2>0$,
independent of the initial value, such that every solution to~\eqref{C} with $(h^0,Q^0)\neq(0,0)$ satisfies
\begin{equation}\label{persistS}
\liminf_{t\to\infty}S_1(t)\ge\underline{S}_1\qquad\text{and}\qquad\liminf_{t\to\infty}S_2(t)\ge\underline{S}_2\,.
\end{equation}
\end{lem}

\begin{proof}
Let $(h^0,Q^0)\neq(0,0)$.
It follows from~\eqref{ccc1x} and the monotonicity~\eqref{r2} of $r$ that
\begin{equation}\label{depthQ}
\int_0^1 r\big(u_1(t,y),u_2(t,y)\big)\,\rd y\le r\big(S_1(t),S_2(t)\big)\le r\big(S_1(t),0\big)\,,\qquad t\ge 0\,.
\end{equation}
We first bound $S_1$ from below. Note from~\eqref{Wbound} that there exists $T_0\ge0$ (depending on the initial value) such that
\begin{equation}\label{barhQ}
\beta h(t)+Q(t)\le \bar c:=\frac{D(k_1+1)S^*}{q}+1\,,\quad t\ge T_0\,.
\end{equation}
We then deduce from~\eqref{depthQ},~\eqref{barhQ}, and~\eqref{C4ax} that
\[
S_1'(t)\ge D\big(S^*-S_1(t)\big)-\bar c \, r\big(S_1(t),0\big)\,,\quad t\ge T_0\,.
\]
The right-hand side is continuous in $S_1(t)$ and equals $DS^*>0$ at  points with $S_1(t)=0$ since $r(0,0)=0$
by~\eqref{r2}, so there is $\underline{S}_1\in(0,S^*)$, independent of the initial value, with
$S_1'(t)\ge\tfrac12 DS^*>0$ whenever $t\ge T_0$ and $S_1(t)\le\underline{S}_1$, 
hence
\begin{equation*}
\liminf_{t\to\infty}S_1(t)\ge\underline{S}_1\,.
\end{equation*}
Regarding $S_2$ we infer from~\eqref{C4bx} that
\[
S_2'(t)\ge-DS_2(t)+r\big(S_1(t),S_2(t)\big)Q(t)\,,\quad t\ge0\,.
\]
Fix $\varepsilon\in\big(0,\min\{\underline{S}_1,\underline Q\}\big)$. Then Lemma~\ref{Tpersist}, the bound
on $S_1$ above, and~\eqref{29c}, ensure that there is
$T_\varepsilon>0$ such that
\[
S_1(t)\ge\underline{S}_1-\varepsilon\,,\quad Q(t)\ge\underline Q-\varepsilon\,,\quad
S_2(t)\le S^*+\varepsilon
\]
for $t\ge T_\varepsilon$
so that, by the monotonicity of $r$ from~\eqref{r2},
\[
S_2'(t)\ge-DS_2(t)+r\big(\underline{S}_1-\varepsilon,S^*+\varepsilon\big)(\underline Q-\varepsilon)\,,\quad t\ge T_\varepsilon\,.
\]
Therefore, $$\liminf_{t\to\infty}S_2(t)\ge \frac{r\big(\underline{S}_1-\varepsilon,S^*+\varepsilon\big)(\underline Q-\varepsilon)}{D}>0$$
which proves the claim.
\end{proof}

Together with the boundedness of Theorem~\ref{T1}, the lower bounds~\eqref{persist}
and~\eqref{persistS} confine every solution with $(h^0,Q^0)\neq(0,0)$ to a compact subset of
$\{h>0, S_1>0, S_2>0, Q>0\}$ in the long-time limit. This separation from the boundary plays a key role in proving the global stability of the nontrivial equilibrium established in Theorem~\ref{Tglobal}.

\section{Global Stability of the  Persistence Equilibrium}\label{sec:7}

Theorem~\ref{T6} establishes the local asymptotic stability of the nontrivial steady state. We now provide conditions under which it is globally asymptotically stable. To this end, we invoke the geometric approach of Li and Muldowney~\cite{LiMuldowney1996}, which excludes nonconstant periodic, homoclinic, and heteroclinic orbits. This criterion has been successfully applied in population dynamics; see, for instance, the example in \cite{LiMuldowney1996} or \cite{Li-Smith-Wang_01}.\\

Recall from~\eqref{29c} that
\[
S_1(t)+S_2(t)-S^*=\big(S_1^0+S_2^0-S^*\big)e^{-Dt}\,,\qquad t\ge0\,.
\]
Hence the plane $\{S_1+S_2=S^*\}$ is forward invariant and attracts every solution at the exponential rate $D$. The long-time dynamics is therefore  governed by the system obtained from~\eqref{UU} by setting
$S_2=S^*-S_1$. Denoting by
\[
R(h,S_1):=\int_0^1 r\big(u[h,S_1,S^*-S_1](y)\big)\,\rd y
\]
the reaction integral of~\eqref{UU} and setting 
$$
{\rm r}(S_1):=r(S_1,S^*-S_1)\,,
$$
this system reads, for $h>0$, $0<S_1<S^*$, and $Q>0$,
\begin{subequations}\label{RED}
\begin{align}
h' &= \big(k_1 R(h,S_1)-b-d(h)\big)h+\tfrac{\alpha}{\beta}Q\,,\label{REDh}\\
S_1' &= D\big(S^*-S_1\big)-{\rm r}(S_1)\,Q-\beta h R(h,S_1)\,,\label{REDS}\\
Q' &= \big(k_1 {\rm r}(S_1)-k_2-\alpha\big)Q+\beta d(h) h\,.\label{REDQ}
\end{align}
\end{subequations}
With $S_2=S^*-S_1$, every solution $(h,S_1,Q)$ of~\eqref{RED} solves~\eqref{C} on this forward invariant
plane, so the persistence established in Theorem~\ref{persistence} holds
along~\eqref{RED} by restriction.
If $h>0$ and
$0<S_1<S^*$, then~\eqref{r2} implies
\begin{subequations}\label{pos}
\begin{equation}\label{r12}
{\rm r}_{12}(S_1):=\dfrac{\rd}{\rd S_1}{\rm r}(S_1)=\partial_1 r(S_1,S^*-S_1)-\partial_2 r(S_1,S^*-S_1)
>0
\end{equation}
and, together with~\eqref{ccc1},
\begin{equation}\label{Rsigns}
0\le R(h,S_1)\le {\rm r}(S_1)\le {\rm r}(S^*)=r(S^*,0)\,,
\end{equation}
while~\eqref{G} and~\eqref{GR} ensure that
\begin{equation}\label{Ldepth}
G_{12}:=\partial_{S_1}R(h,S_1)\ge 0\,,\qquad H:=\partial_h R(h,S_1)\le 0\,.
\end{equation}
\end{subequations}
Note then that the Jacobian $J=J(h,S_1,Q)$ of~\eqref{RED} at the general state~$(h,S_1,Q)\in\R_+^3$ is given by
{\small
\begin{equation*}
J=\left(\begin{matrix}
k_1 R(h,S_1)+\big(k_1H-d'(h)\big)h-b-d(h) & k_1 hG_{12} & \frac{\alpha}{\beta}\\[4pt]
-\beta\big(R(h,S_1)+hH\big) & -D-{\rm r}_{12}(S_1)Q-\beta hG_{12} & -{\rm r}(S_1)\\[4pt]
\beta\big(d(h)+d'(h)h\big) & k_1 {\rm r}_{12}(S_1)Q & k_1 {\rm r}(S_1)-k_2-\alpha
\end{matrix}\right) .
\end{equation*}}
Let 
\begin{equation}\label{mu1}
\mu_1(A):=\max_{1\le j\le3}\Big(a_{jj}+\sum_{\substack{1\le i\le3\\ i\ne j}}|a_{ij}|\Big)
\end{equation}
be the Lozinski\u{\i} measure of a $(3\times3)$--matrix $A=(a_{ij})_{1\le i,j\le3}$ with respect to the $\ell^1$--norm. The geometric approach of Li and Muldowney \cite{LiMuldowney1996,Li-Smith-Wang_01} is based on a Bendixson criterion, which requires the construction of a nonsingular matrix $P=\big(p_{ij}(h,S_1,Q)\big)_{1\le i,j\le3}$ such that
\begin{equation}\label{mu1X}
\mu_1\Big(\Big(\frac{\rd}{\rd t}p_{ij}\Big)_{1\le i,j\le3}P^{-1}+PJ^{[2]}P^{-1}\Big)<0
\qquad\text{ in }\quad\{h>0, 0<S_1<S^*, Q>0\}\,,
\end{equation}
where $\tfrac{\rd}{\rd t}p_{ij}$ is the derivative of the entry $p_{ij}=p_{ij}(h,S_1,Q)$ along~\eqref{RED} and
\begin{equation*}
J^{[2]}=\left(\begin{matrix}
J_{11}+J_{22} & J_{23} & -J_{13}\\
J_{32} & J_{11}+J_{33} & J_{12}\\
-J_{31} & J_{21} & J_{22}+J_{33}
\end{matrix}\right).
\end{equation*}
is the second additive compound matrix of $J=(J_{ij})_{1\le i,j\le3}$ (see \cite{Muldowney1990} or \cite[Appendix]{Li-Smith-Wang_01}).
We take the constant diagonal matrix
\begin{equation*}
P:=\begin{pmatrix} k_1 & 0 & 0\\[2pt] 0 & 1 & 0\\[2pt] 0 & 0 & \frac{k_1}{\beta} \end{pmatrix},
\end{equation*}
whose entries do not depend on $(h,S_1,Q)$, so that $\tfrac{\rd}{\rd t}p_{ij}=0$ and condition~\eqref{mu1X} reduces to 
$\mu_1\big(PJ^{[2]}P^{-1}\big)<0$. Since $P$ is diagonal, we have $\big(PJ^{[2]}P^{-1}\big)_{ij}=(p_i/p_j)\big(J^{[2]}\big)_{ij}$ for
$1\le i,j\le3$ with  $(p_1,p_2,p_3)=(k_1,1,k_1/\beta)$. Consequently,
\[
PJ^{[2]}P^{-1}=\begin{pmatrix}
J_{11}+J_{22} & -k_1{\rm r}(S_1) & -\alpha\\[4pt]
{\rm r}_{12}(S_1)Q & J_{11}+J_{33} & \beta hG_{12}\\[4pt]
-\big(d(h)+d'(h)h\big) & -k_1\big(R(h,S_1)+hH\big) & J_{22}+J_{33}
\end{pmatrix}.
\]
By~\eqref{pos} and $d(h)+d'(h)h>0$, every off--diagonal entry of this matrix has a definite sign except for
 $-k_1\big(R(h,S_1)+hH\big)$, which alone keeps its absolute value when computing $\mu_1\big(PJ^{[2]}P^{-1}\big)$.  It thus follows from~\eqref{mu1} that
\begin{equation}\label{colsum}
\begin{aligned}
\mu_1\big(PJ^{[2]}P^{-1}\big)
&=\max\left\{
\begin{aligned}
&k_1 R(h,S_1)+\big(k_1 H-\beta G_{12}\big)h-b-D,\\
&k_1\big(R(h,S_1)+hH\big)+k_1\big|R(h,S_1)+hH\big|+2k_1 {\rm r}(S_1)\\
&\qquad-\alpha-b-d(h)-d'(h)h-k_2,\\
&k_1 {\rm r}(S_1)-{\rm r}_{12}(S_1)\,Q-D-k_2
\end{aligned}
\right\}.
\end{aligned}
\end{equation}

With these preparations in place, we can prove the global stability of the persistence equilibrium.

\begin{thm}\label{Tglobal}
Assume~\eqref{GAshooting},~\eqref{ZZ},~\eqref{d1}, and
\begin{equation}\label{C1C2}
k_1 r(S^*,0)<b+D\qquad\text{and}\qquad 4k_1 r(S^*,0)<\alpha+b+k_2\,.
\end{equation}
Then the persistence equilibrium $(h_\star,S_{\star,1},S_{\star,2},Q_\star)$ of Theorem~\ref{T5} is globally
asymptotically stable; that is, it attracts every solution of~\eqref{C} with $(h^0,Q^0)\neq(0,0)$.
\end{thm}

\begin{proof} 
As pointed out above, we may restrict to the system~\eqref{RED} for which we verify the four hypotheses of~\cite[Theorem~3.5]{LiMuldowney1996} with respect to
the open, simply connected set 
\[
\mathcal D:= \{h>0,0<S_1<S^*,Q>0\}\, .
\] 
Along~\eqref{RED}, we obtain for any initial value in $\mathcal D$ that $h$ and $Q$
are bounded above due to~\eqref{barhQ}, say by $\overline{h}$ and $\overline{Q}$, respectively, and below by~\eqref{persist}, and 
with $S_2=S^*-S_1$ we have 
\[
\underline{S}_1\le\liminf_{t\to\infty}S_1(t)\le\limsup_{t\to\infty}S_1(t)\le S^*-\underline{S}_2\,.
\]
Thus, its $\omega$-limit set lies in the compact subset
\[
\mathcal{K}:=[\underline h,\overline h]\times[\underline{S}_1,S^*-\underline{S}_2]\times[\underline Q,\overline Q]
\]
of $\mathcal D$. As $\mathcal K$ is compact and, by~\eqref{persist} and~\eqref{persistS}, bounded away from
$\partial\mathcal D$, every solution eventually enters an arbitrarily small neighborhood of $\mathcal K$ in
$\mathcal D$, which is the compact absorbing set required in assumption~$(H_2)$ of~\cite{LiMuldowney1996}
(cf.~the persistence attractor of~\cite[Theorem~5.7]{SmithThieme2011}).
By Theorem~\ref{T5}, the
point $(h_\star,S_{\star,1},Q_\star)$ is the only equilibrium in $\mathcal D$, and by Theorem~\ref{T6} it is locally
asymptotically stable. This is condition $(H_3)$ of \cite{LiMuldowney1996}.

Let $P$  and $J$ be as above. It then remains to bound the three column sums~\eqref{colsum}, which is the last condition of~\cite{LiMuldowney1996}. Since~\eqref{pos} and~\eqref{d} entail that
\begin{align*}
R(h,S_1)&\le r(S^*,0)\,,\qquad(k_1 H-\beta G_{12})h\le0\,,\\
 -{\rm r}_{12}(S_1)Q&\le0\,,\qquad -(d(h)+d'(h)h)\le0\,,
\end{align*}
and
\[
k_1\big(R(h,S_1)+hH\big)+k_1\big|R(h,S_1)+hH\big|=2k_1\max\big\{R(h,S_1)+hH, 0\big\}\le 2k_1 r(S^*,0)\,,
\]
the three entries of~\eqref{colsum} satisfy
\[
k_1 R(h,S_1)+\big(k_1 H-\beta G_{12}\big)h-b-D\le k_1 r(S^*,0)-b-D\,,
\]
and 
\begin{align*}
&k_1\big(R(h,S_1)+hH\big)+k_1\big|R(h,S_1)+hH\big|+2k_1 {\rm r}(S_1)\\
&\qquad -\alpha-b-d(h)-d'(h)h-k_2\le 4 k_1 r(S^*,0)-\alpha-b-k_2\,,
\end{align*}
and 
\[
k_1 {\rm r}(S_1)-{\rm r}_{12}(S_1) Q-D-k_2\le k_1 r(S^*,0)-D-k_2\,.
\]
It then follows that
\[
\mu_1\big(PJ^{[2]}P^{-1}\big)\le-\theta<0\qquad\text{on}\quad\mathcal D\,,
\]
where
\[
\theta:=\min\big\{b+D-k_1 r(S^*,0), \alpha+b+k_2-4k_1 r(S^*,0), D+k_2-k_1 r(S^*,0)\big\}>0
\]
 due to~\eqref{C1C2} and~\eqref{Z1x}.
 Consequently, ~\cite[Theorem~3.5]{LiMuldowney1996} implies that the equilibrium $(h_\star,S_{\star,1},Q_\star)$  of~\eqref{RED} is globally asymptotically stable in $\mathcal{D}$.
By the exponential convergence $S_1(t)+S_2(t)\to S^*$ of~\eqref{29c}, the $(h,S_1,Q)$--dynamics
from~\eqref{C} is asymptotically autonomous with limit system~\eqref{RED}, and~\cite[Theorem~4.1]{Thieme1992}
now yields the global asymptotic stability of $(h_\star,S_{\star,1},S_{\star,2},Q_\star)$ for~\eqref{C}.
\end{proof}

 In Appendix~\ref{AppN} we provide an example with a regime, where conditions~\eqref{ZZ},~\eqref{d1}, and~\eqref{C1C2} are satisfied.

\begin{figure}[h]
\centering
\scalebox{0.85}{%
\begin{tikzpicture}[>={Stealth[length=2.2mm]}, line cap=round, line join=round,
  orb/.style={blue!55!black, line width=0.9pt, postaction={decorate},
    decoration={markings, mark=at position #1 with {\arrow{Stealth}}}}, orb/.default=0.62,
  offorb/.style={blue!48!black, line width=0.9pt, densely dashed, postaction={decorate},
    decoration={markings, mark=at position #1 with {\arrow{Stealth}}}}, offorb/.default=0.62,
  stab/.style={circle, fill=black, draw=black, inner sep=0pt, minimum size=3.6pt},
  unst/.style={circle, fill=white, draw=red!62!black, line width=0.9pt, inner sep=0pt, minimum size=4pt}]
  \shade[top color=gray!3, bottom color=gray!14] (0,0) -- (8,0) -- (10,2.4) -- (2,2.4) -- cycle;
  \draw[gray!40, line width=0.4pt] (0,0) -- (2,2.4);
  \draw[gray!55, line width=0.6pt] (2,2.4) -- (10,2.4) -- (8,0);
  \node[gray!72, font=\scriptsize\itshape, rotate=50.19, anchor=center] at (1.4,1.2) {$\{S_1{+}S_2{=}S^*\}$};
  \draw[red!62!black, line width=0.9pt] (0,0) -- (8,0);
  \node[red!62!black, font=\scriptsize, anchor=north] at (5.6,-0.03) {$\{h=0\}$};
  \node[unst, label={[red!62!black, font=\scriptsize]below:$X_0$}] (E0) at (2.05,0) {};
  \draw[->, red!62!black, line width=0.7pt] (E0) -- ++(52:1.0);
  \begin{scope}
    \clip (0,0) -- (8,0) -- (10,2.4) -- (2,2.4) -- cycle;
    \fill[teal!9] (3.58,0.815) -- (6.08,0.815) -- (6.72,1.585) -- (4.22,1.585) -- cycle;
  \end{scope}
  \draw[teal!55!black, densely dashed, line width=0.7pt]
        (3.58,0.815) -- (6.08,0.815) -- (6.72,1.585) -- (4.22,1.585) -- cycle;
  \node[teal!50!black, font=\scriptsize] at (6.95,1.6) {$\mathcal{K}$};
  \draw[orb=0.62] (2.6,1.72) .. controls (3.5,1.38) and (4.25,1.27) .. (4.88,1.23);
  \draw[orb=0.62] (3.3,0.5)  .. controls (4.15,0.86) and (4.62,1.04) .. (5.02,1.11);
  \draw[orb=0.62] (7.6,0.6)  .. controls (6.7,0.86)  and (5.95,1.04) .. (5.3,1.11);
  \draw[orb=0.62] (8.4,1.72) .. controls (7.3,1.38)  and (6.2,1.27)  .. (5.42,1.23);
  \draw[offorb=0.62] (3.0,2.95) .. controls (3.85,2.05) and (4.62,1.62) .. (4.9,1.42);
  \draw[offorb=0.62] (4.55,3.2) .. controls (4.78,2.2) and (5.02,1.66) .. (5.1,1.43);
  \draw[offorb=0.62] (7.2,2.95) .. controls (6.4,2.05) and (5.68,1.62) .. (5.4,1.42);
  \node[stab, label={[font=\scriptsize]below right:$X_\star$}] (Es) at (5.15,1.2) {};
  \draw[->, gray!75, line width=0.55pt] (-0.5,-0.35) -- (1.15,-0.35) node[right=-1pt, black, font=\scriptsize]{$S_1$};
  \draw[->, gray!75, line width=0.55pt] (-0.5,-0.35) -- (0.12,0.42) node[above left=-3pt, black, font=\scriptsize]{$h$};
\end{tikzpicture}}%
\caption{ Schematic of the global dynamics on the invariant plane (a projection to $(h,S_1)$, with
$S_2=S^*-S_1$ and $Q$ suppressed). 
Any deviation $S_1^0+S_2^0-S^*$ collapses at an exponential rate onto the forward invariant plane $\{S_1+S_2=S^*\}$ for~\eqref{RED}. 
Solid curves represent orbits of~\eqref{RED} with $S_1^0+S_2^0=S^*$
while dashed curves represent solutions of~\eqref{UU} with $S_1^0+S_2^0\neq S^*$, which approach the plane. On the face $\{h=0\}$ the trivial equilibrium $X_0=(0,S^*,0,0)$ is unstable and a uniform weak
repeller. By uniform persistence (Theorem~\ref{persistence}), every solution
with  $(h^0,Q^0)\neq(0,0)$ has its $\omega$-limit set in the compact set $\mathcal{K}\subset\mathcal D$ and
converges  to the nontrivial equilibrium $X_\star=(h_\star,S_{\star,1},S_{\star,2},Q_\star)$ (Theorem~\ref{Tglobal}).
}
\label{fig:globaldynamics}
\end{figure}
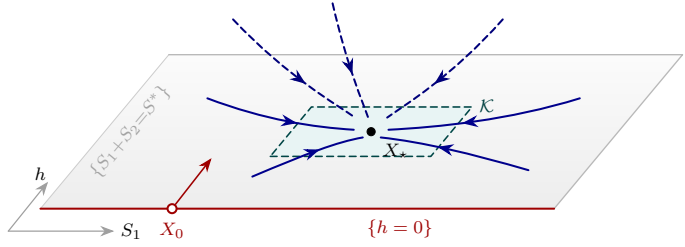

\begin{appendix}

\nequation
\aequation

\section{Proof of Theorem~\ref{T6}}\label{App}

To finish off the proof of Theorem~\ref{T6} it remains to verify that all eigenvalues of the $(3\times 3)$-matrix $A'=(a_{ij})_{1\le i,j\le 3}$ in \eqref{MA} have negative real parts. According to the Routh-Hurwitz-Criterion this is the case provided that
\begin{align}\label{mpos}
m_0>0\,,\quad m_1>0\,,\quad m_2>0\,,\quad m_1m_2-m_0>0\,,
\end{align}
where 
\begin{align*}
&m_0:=-\det(A')\,,\qquad m_1:=-\mathrm{trace}(A')\,,\\
&m_2:=\det\left(\begin{matrix}
a_{11} & a_{12} \\
a_{21} & a_{22}
\end{matrix}\right)
+\det\left(\begin{matrix}
a_{11} & a_{13} \\
a_{31} & a_{33}
\end{matrix}\right)
+\det
\left(\begin{matrix}
a_{22} & a_{23} \\
a_{32} & a_{33}
\end{matrix}\right)\,.
\end{align*}
We verify~\eqref{mpos} in four steps.\\

\noindent{\bf (i)} As for $m_0$ one computes
\begin{align*}
m_0=&
-\Big( k_1 r_{12}(S_\star) r(S_\star)   Q_\star   +   D\Delta +r_{12}(S_\star) Q_\star  \big[\Delta-\alpha\big]\Big)k_1  h_\star  H\\
&+ \Big(\Delta  G_{12}  \beta  h_\star+Q_\star \alpha  r_{12}(S_\star) \Big)k_1  R\\
& +\Big(k_1 \left(G_{12}  \beta  h_\star+r_{12}(S_\star) Q_\star \right) r (S_\star )+\left[\Delta -\alpha \right] \left(\beta  h_\star G_{12} +r_{12}(S_\star) Q_\star +D\right)\Big) h_\star d'(h_\star) \\ 
&+\frac{k_1}{\Delta} \Big(\Delta  G_{12}  \beta  h_\star+Q_\star \alpha  r_{12}(S_\star) \Big) d(h_\star) r (S_\star )
\end{align*}
and thus $m_0>0$ since each term is nonnegative owing to~\eqref{Z}, where we use ~\eqref{Delta1} for the terms in square brackets, the last term being positive since $r_{12}(S_\star)>0$, $r(S_\star)>0$, $Q_\star>0$, and $d(h_\star)>0$.\\

\noindent{\bf (ii)} As for $m_1$ one computes
\begin{align*}
m_1=&
 - k_1 h_\star H +\frac{\alpha  d(h_\star)}{\Delta}+d'(h_\star)  h_\star+\beta  h_\star G_{12} +r_{12}(S_\star) Q_\star +D+\Delta
\end{align*}
and thus $m_1>0$ since each term is nonnegative owing to~\eqref{Z} and ~\eqref{Delta1}, the summands $D$ and $\Delta$ being positive.\\

\noindent{\bf (iii)} 
As for $m_2$ one computes
\begin{align*}
m_2=&-\Big(  \Delta +  r_{12}(S_\star) Q_\star +D\Big) k_1 h_\star H +  \beta k_1 G_{12} R h_\star\\
& +\Big(\beta  h_\star G_{12} +r_{12}(S_\star) Q_\star +D\Big)\Big(h_\star d'(h_\star) +\frac{\alpha d(h_\star)}{\Delta} +\Delta\Big)\\
& +[\Delta -\alpha] h_\star d'(h_\star)+ k_1r (S_\star )  r_{12}(S_\star) Q_\star
 \end{align*}
and thus $m_2>0$ since each term is nonnegative owing to~\eqref{Z}, where we use ~\eqref{Delta1}  for the term in square brackets, the third summand being bounded below by $D\Delta>0$.\\

\noindent{\bf (iv)} As for $m_3:=m_1m_2-m_0$ one computes, using~\eqref{Q},
\begin{align*}
\Delta m_3=&
\left(- h_\star\Delta   k_1 H +d'(h_\star)  h_\star\Delta +\Delta^{2} +  \alpha  d(h_\star)\right) \mathrm{D}^{2} +D\Delta   k_1^{2} h_\star^{2} H^{2}\\
&\ +D\left(- k_1  h_\star \Delta   H + k_1  \Delta   R +2 \Delta   d'(h_\star)  h_\star  +2\Delta^{2}+  2\alpha  d(h_\star)\right) \beta   h_\star G_{12}\\
&\ -2D\Big(    \Delta  d'(h_\star)h_\star + \beta  r_{12}(S_\star) d(h_\star) h_\star+\Delta^{2}+  \alpha  d(h_\star)  \Big) k_1  h_\star H +D\Delta  d'(h_\star)^{2} h_\star^{2}\\
& \ +2D\Big( \beta r_{12}(S_\star)  d(h_\star) h_\star +\Delta^2+\alpha d(h_\star)\Big)  d'(h_\star)h_\star\\ 
& \ + D\beta \left( k_1  r(S_\star) +2\Delta+\frac{2\alpha d(h_\star)}{\Delta}\right)  d(h_\star) h_\star r_{12}(S_\star)+\frac{D}{\Delta}\left(\Delta^{2}+  \alpha  d(h_\star)\right)^2 \\
&\ +\Big( k_1   \Delta    R +\Delta  h_\star d'(h_\star) +\alpha  d(h_\star)+\Delta^2 \Big) h_\star^{2}\beta^{2} G_{12}^{2}\\
&\ -\left(k_1  \Delta  R +   \Delta  h_\star d'(h_\star) +\beta d(h_\star) h_\star  r_{12}(S_\star) +2\Delta^{2} +\alpha  d(h_\star)\right)\beta k_1 h_\star^{2} G_{12}   H\\
&\  +\Big(\Delta  d'(h_\star)h_\star + \beta d(h_\star)r_{12}(S_\star)h_\star+\alpha  d(h_\star)\Big) k_1 \beta h_\star    G_{12} R
+\beta\Delta  h_\star^{3}  G_{12} d'(h_\star)^{2}\\
&\ + 2\Big( \beta r_{12}(S_\star) h_\star +\alpha  \Big)\beta d'(h_\star)d(h_\star) h_\star^{2}G_{12} + \big[2\Delta-k_1 r(S_\star)\big]  \beta \Delta d'(h_\star) h_\star^{2}G_{12}\\
&\ +\left( k_1  r(S_\star) +2\Delta+\frac{2\alpha d(h_\star)}{\Delta}\right) \beta^2 d(h_\star) h_\star^{2} r_{12}(S_\star)G_{12}\\
&\ + \big[2\alpha-k_1 r(S_\star)\big] \beta \Delta d(h_\star) h_\star G_{12}+\left(\Delta^{3}+\frac{\alpha^2 d(h_\star)^2}{\Delta}\right)\beta h_\star G_{12}\\
& +\Big(\beta  d(h_\star) h_\star  r_{12}(S_\star) +  \Delta^{2}\Big)k_1^{2} h_\star^{2} H^{2}\\
&-\Bigg\{\left(2 \beta  d(h_\star) h_\star^{3}  k_1 r_{12}(S_\star) +\big[2\Delta-\alpha\big] k_1 \Delta h_\star^2    \right) d'(h_\star) +\frac{\beta^{2} d(h_\star)^{2} h_\star^{3}  k_1 }{\Delta} r_{12}(S_\star)^{2}\\
&\qquad + \Bigg(2\Delta+\frac{2\alpha d(h_\star)}{\Delta}+\alpha \Bigg) k_1\beta d(h_\star)h_\star^2 r_{12}(S_\star) +\left(\Delta^{2}+\alpha  d(h_\star)\right)  k_1  h_\star\Delta \Bigg\} H\\
& +\Big(\beta  d(h_\star) h_\star^{3} r_{12}(S_\star) +\big[\Delta-\alpha\big]  \Delta h_\star^2  \Big) d'(h_\star)^{2}\\
&+\Bigg\{\frac{\beta^{2} d(h_\star)^{2} h_\star^{3} r_{12}(S_\star)^{2}}{\Delta}+2\Bigg(\frac{\alpha d(h_\star)}{\Delta} +\Delta\Bigg)\beta d(h_\star)h_\star^2 r_{12}(S_\star)\\
&\qquad\qquad +\big[\Delta-\alpha\big]\big(\Delta^2+\alpha d(h_\star)\big)h_\star\Bigg\} d'(h_\star)\\
& +\frac{1}{\Delta^2}\Big( k_1 \Delta   r(S_\star) +\Delta^{2}+\alpha  d(h_\star)\Big) \beta^2  d(h_\star)^2 h_\star^2  r_{12}(S_\star)^{2}\\
&
\\
&+ \left(\Delta^{2}+2\alpha d(h_\star)+\frac{\alpha^2d(h_\star)^2}{\Delta^2}\right)\beta d(h_\star) h_\star r_{12}(S_\star)\\
& +\Big[\Delta r(S_\star)-\alpha R\Big] k_1\beta d(h_\star)h_\star r_{12}(S_\star)
\end{align*}
and thus $m_3>0$ since each term is nonnegative owing to~\eqref{Z}, where we use~\eqref{d1} and~\eqref{Delta1} for all square-bracketed factors except $[\Delta r(S_\star)-\alpha R]$, for which we note that
$$
\Delta r(S_\star)-\alpha R \ge \big[\Delta -\alpha\big] r(S_\star)\ge 0
$$
since $R\le r(S_\star)$ according to~\eqref{GH3}. 
 
Consequently, all eigenvalues of the linearization have a negative real part. This proves Theorem~\ref{T6}.

\section{Simulations}\label{AppN}
We illustrate the stability of the washout equilibrium (Theorem~\ref{P2}), the existence of the nontrivial equilibrium (Theorems~\ref{T5} and~\ref{T6}), and its global stability (Theorem~\ref{Tglobal}) by providing simulations for~\eqref{C} using the thermodynamic growth function from~\eqref{R:kinetics} and the linear detachment $d(h)=\delta h$. Multiplying out, \eqref{R:kinetics} reads 
\[
g(S_1,S_2)=\frac{\mu[S_1-\Gamma S_2]_+}{K+S_1}\, ,
\]
and serves here to illustrate the analysis carried out for the general $r$.

We use the parameters from Table~\ref{tab:sim}: the kinetic and transport constants follow~\cite{GHE21}, while the operating parameters $(\alpha,D,b)$ and the influent $S^*=1$ are illustrative. At $S^*=1$, the washout growth $
k_1r(S^*,0)=g(S^*,0)=\frac{\mu S^*}{K+S^*}=2.574$
exceeds $\max\{2\alpha,k_2\}=2.005$ in Case~1, so~\eqref{ZZ} and~\eqref{d1} fail and the washout equilibrium is stable (by~\eqref{S1},~\eqref{S2}), while it stays below $\min\{2\alpha,k_2\}=2.6$ in Cases~2 and~3, so~\eqref{ZZ} and~\eqref{d1} hold and the persistence equilibrium appears. The elevated $b=3$ of Case~1 models strong biofilm lysis rather than baseline decay.

\begin{table}[h]
\centering
\small
\setlength{\tabcolsep}{5pt}
\renewcommand{\arraystretch}{1.3}
\begin{tabular}{clccc}
\hline
\textbf{Symbol} & \textbf{Description} & \textbf{Case 1} & \textbf{Case 2} & \textbf{Case 3}\\ \hline
$\mu$ & maximal growth rate & \multicolumn{3}{c}{$2.6$}\\
$K$ & half-saturation constant & \multicolumn{3}{c}{$10^{-2}$}\\
$\Gamma$ & thermodynamic inhibition factor & \multicolumn{3}{c}{$1.503$}\\
$\kappa_1$ & diffusion coefficient of propionate & \multicolumn{3}{c}{$8.2\times10^{-5}$}\\
$\kappa_2$ & diffusion coefficient of acetate & \multicolumn{3}{c}{$9.4\times10^{-5}$}\\
$\delta$ & detachment coefficient & \multicolumn{3}{c}{$2280$}\\
$k_1$ & yield constant & \multicolumn{3}{c}{$1.19\times10^{-4}$}\\
$\beta$ & biofilm-bulk conversion factor & \multicolumn{3}{c}{$46.6$}\\
$S^*$ & influent propionate concentration & \multicolumn{3}{c}{$1$}\\ \hline
$\alpha$ & attachment rate & $1$ & $1.3$ & $1.3$\\
$D$ & dilution rate & $2$ & $5$ & $10$\\
$b$ & biomass loss rate & $3$ & $0.005$ & $0.005$\\
$k_2$ & removal rate of $Q$ ($=D+0.005$) & $2.005$ & $5.005$ & $10.005$\\ \hline
\end{tabular}
\vspace{0.2cm}
\caption{Parameters used in the numerical simulations. The shared kinetic and transport constants are taken from~\cite{GHE21}, and the operating parameters $(\alpha,D,b)$ are varied across the three cases.}\label{tab:sim}
\end{table}

In Case~1 both~\eqref{S1} and~\eqref{S2} hold, so the washout equilibrium $(h_\star,S_{\star,1},S_{\star,2},Q_\star)=(0,S^*,0,0)$ is locally asymptotically stable and the solution converges to it (Figure~\ref{fig:case1}). In Case~2 both~\eqref{ZZ} and~\eqref{d1} hold while~\eqref{C1C2} fails, and the solution converges to the numerically determined  persistence equilibrium
\[
(h_\star,S_{\star,1},S_{\star,2},Q_\star)=\big(5.0\times10^{-5},\,0.616,\,0.384,\,4.3\times10^{-5}\big)\,,\qquad S_{\star,1}+S_{\star,2}=S^*\,,
\]
as predicted by Theorems~\ref{T5} and~\ref{T6} (Figure~\ref{fig:case2}). In Case~3 the stronger condition~\eqref{C1C2} holds as well, and the solutions from five distinct initial conditions converge to the same interior state
\[
(h_\star,S_{\star,1},S_{\star,2},Q_\star)=\big(6.6\times10^{-5},\,0.629,\,0.371,\,4.2\times10^{-5}\big)\,,\qquad S_{\star,1}+S_{\star,2}=S^*\,,
\]
in agreement with the global stability of Theorem~\ref{Tglobal} (Figure~\ref{fig:case3}).

\begin{figure}[H]
\centering
\includegraphics[width=0.825\textwidth]{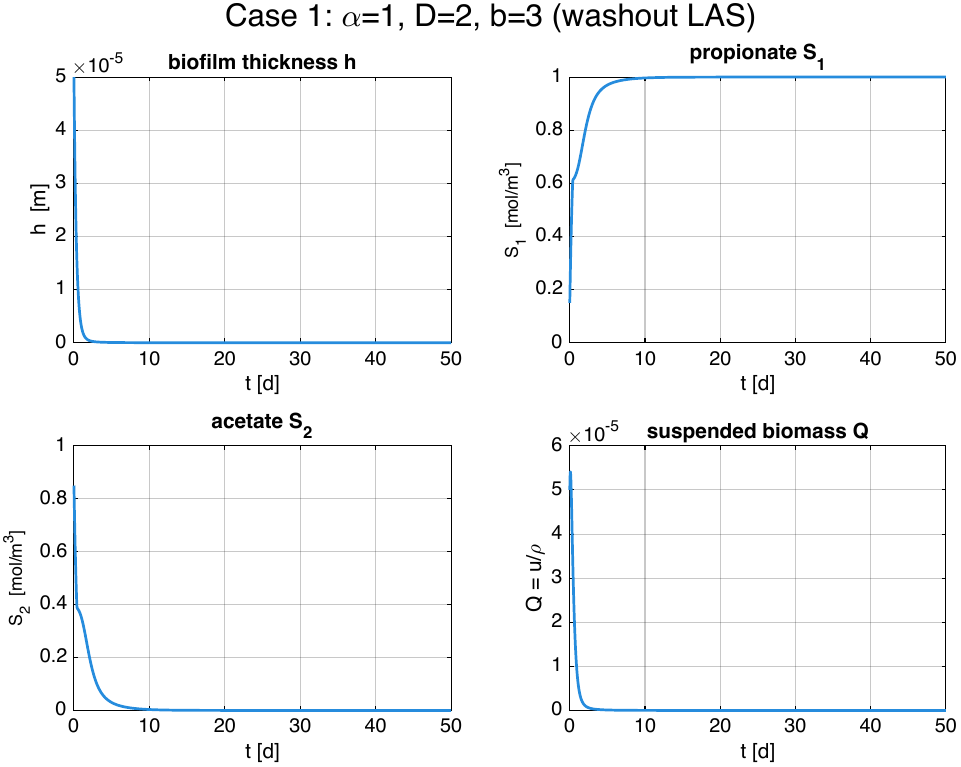}
\caption{Case~1. Convergence to the washout equilibrium of Theorem~\ref{P2} under~\eqref{S1},~\eqref{S2}.}\label{fig:case1}
\end{figure}


\begin{figure}[h]
\centering
\includegraphics[width=0.825\textwidth]{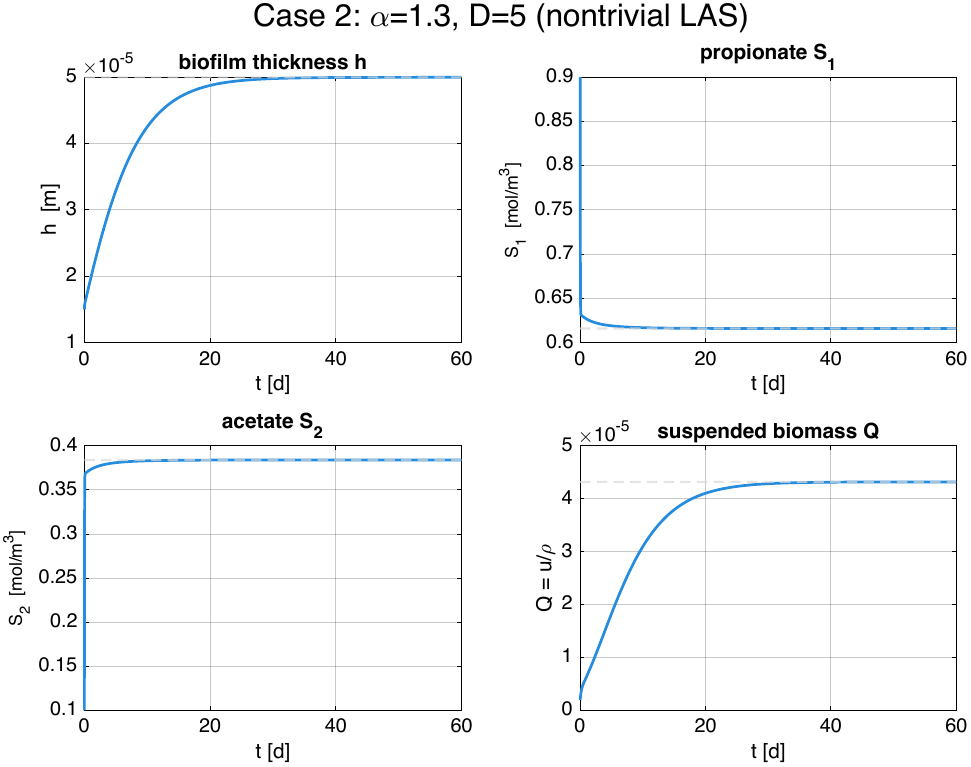}
\caption{Case~2. Convergence to the nontrivial equilibrium of Theorems~\ref{T5} and~\ref{T6} under~\eqref{ZZ},~\eqref{d1}.}\label{fig:case2}
\end{figure}


\begin{figure}[H]
\centering
\includegraphics[width=0.65\textwidth]{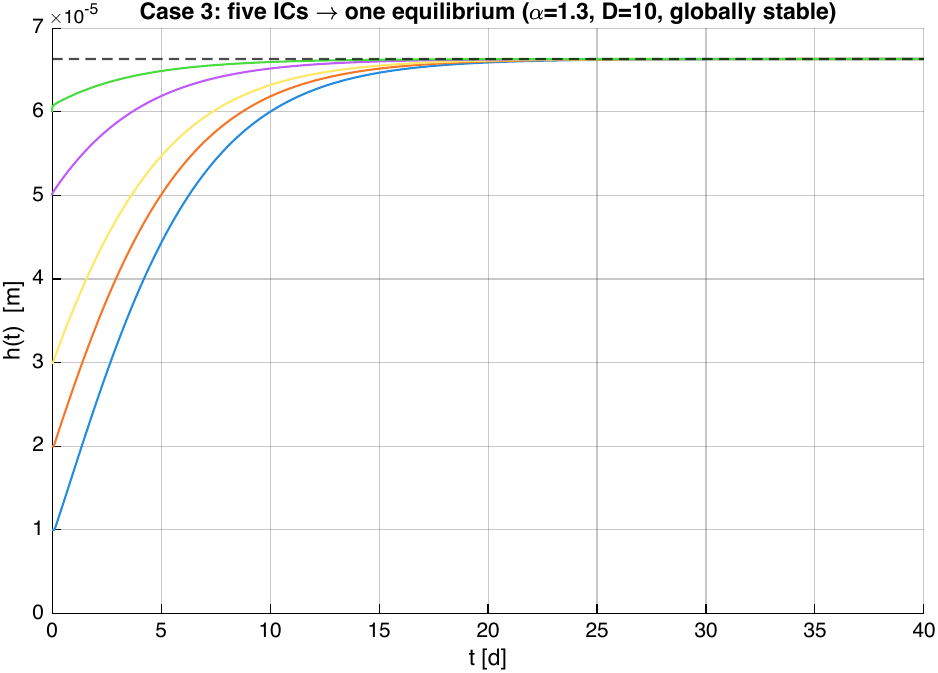}
\caption{Five $h$-trajectories from different initial conditions showing that all solutions converge to the unique nontrivial equilibrium $(h_\star,S_{\star,1},S_{\star,2},Q_\star)$, illustrating the global asymptotic stability of Theorem~\ref{Tglobal} under~\eqref{ZZ},~\eqref{d1},~\eqref{C1C2}.}\label{fig:case3}
\end{figure}

\end{appendix}

\bibliographystyle{siam} 
\bibliography{BiofilmLiterature}

\end{document}